\documentclass{amsart}

\input xy

\xyoption{all}

\usepackage{amssymb,amsbsy,amsthm,amsmath,graphicx,epsfig}

\newtheorem{thm}{Theorem}[section]
\newtheorem*{thm*}{Theorem}
\newtheorem{theorem}{Theorem}[section]

\newtheorem{lemma}[thm]{Lemma}

\newtheorem{proposition}[thm]{Proposition}

\newtheorem{definition}[thm]{Definition}

\theoremstyle{remark}

\newtheorem{remark}{Remark}[section]

\newtheorem{example}[remark]{Example}

\newcommand{\id}{\mathrm{id}}

\newcommand{\N}{\mathbf{N}}

\newcommand{\R}{\mathbf{R}}
\newcommand{\C}{\mathbf{C}}

\newcommand{\Z}{\mathbf{Z}}

\newcommand{\eps}{\varepsilon}

\renewcommand{\hat}[1]{\widehat{#1}}

\newcommand{\HK}{\operatorname{HK}}
\newcommand{\HP}{\operatorname{HP}}
\newcommand{\unf}{{\operatorname{unf}}}
\newcommand{\nil}{{\operatorname{nil}}}
\newcommand{\sml}{{\operatorname{sml}}}

\newcommand{\E}{\mathbb{E}}

\begin{document}

\title{Large values of the Gowers-Host-Kra seminorms}
\author{Tanja Eisner \and Terence Tao}
\thanks{T.E. is supported by the European Social Fund
and by the Ministry of Science, Research and the Arts
Baden-W\"urttemberg. T.T. is supported by NSF grant DMS-0649473
and a grant from the Macarthur Foundation.}
\address{Universit\"at  T\"ubingen \and University of California, Los Angeles}
\date{}

\begin{abstract}   The \emph{Gowers uniformity norms} $\|f\|_{U^k(G)}$ of a function $f: G \to \C$ on a finite additive group $G$, together with the slight variant $\|f\|_{U^k([N])}$ defined for functions on a discrete interval $[N] := \{1,\ldots,N\}$, are of importance in the modern theory of counting additive patterns (such as arithmetic progressions) inside large sets.  Closely related to these norms are the \emph{Gowers-Host-Kra seminorms} $\|f\|_{U^k(X)}$ of a measurable function $f: X \to \C$ on a measure-preserving system $X = (X, {\mathcal X}, \mu, T)$.  Much recent effort has been devoted to the question of obtaining necessary and sufficient conditions for these Gowers norms to have non-trivial size (e.g. at least $\eta$ for some small $\eta > 0$), leading in particular to the inverse conjecture for the Gowers norms, and to the Host-Kra classification of characteristic factors for the Gowers-Host-Kra seminorms.

In this paper we investigate the near-extremal (or ``property
testing'') version of this question, when the Gowers norm or
Gowers-Host-Kra seminorm of a function is almost as large as it can be
subject to an $L^\infty$ or $L^p$ bound on its magnitude.  Our main
results assert, roughly speaking, that this occurs if and only if $f$
behaves like a polynomial phase, possibly localised to a subgroup of
the domain; these results can be viewed as higher-order analogues of a
classical result of Russo~\cite{russo} and Fournier~\cite{fournier}, and are also related to the polynomiality testing results over finite fields of Blum-Luby-Rubinfeld~\cite{blr} and Alon-Kaufman-Krivelevich-Litsyn-Ron~\cite{akklr}.  We investigate the situation further for the $U^3$ norms, which are associated to $2$-step nilsequences, and find that there is a threshold behaviour, in that non-trivial $2$-step nilsequences (not associated with linear or quadratic phases) only emerge once the $U^3$ norm is at most $2^{-1/8}$ of the $L^\infty$ norm.
\end{abstract}

\maketitle


\section{Introduction}

\subsection{The Gowers norms and Gowers-Host-Kra seminorms}

The purpose of this paper is to investigate functions which have an exceptionally large Gowers norm $U^k(G)$ or Gowers-Host-Kra seminorm $U^k(X)$.  To do this we first recall the definitions for these norms.

We begin with the Gowers norms $U^k(G)$.  These norms are usually defined on a finite abelian group $G = (G,+)$, such as a cyclic group $\Z/N\Z$ or a finite field vector space ${\mathbb F}_p^n$, but our results are most naturally stated in the more general setting\footnote{The Gowers uniformity norms can also be defined on some non-abelian groups, and in particular in finite groups or nilpotent Lie groups; this is implicit in \cite{host-kra}.  However, we will not consider such generalisations here.  We also note that a slightly different Gowers-type seminorm was constructed on $\ell^\infty(\Z)$ by Host and Kra \cite{host-kra-inf}, which we will also not study directly here.} of a locally compact abelian (LCA) group $G = (G,+)$ equipped with a non-trivial Haar measure, $\mu$, i.e. non-trivial translation-invariant Radon measure on the Borel $\sigma$-algebra ${\mathcal B}$ of $G$.  In the case of a compact abelian group, we will require that the Haar measure is \emph{normalised}, so that $\mu(G)=1$, but of course we cannot require this in the non-compact case.  In particular, if $G$ is a finite abelian group, then the Haar measure is the normalised counting measure $\mu(E) := |E|/|G|$.

For technical reasons it is convenient to assume that the group $G$ is second countable; this condition may almost certainly be removed, but we will not do so here as one has the second countability axiom in all known applications of the Gowers norms.  Thus, in this paper, all locally compact abelian groups are implicitly understood to be second countable.  In particular, their Borel $\sigma$-algebra ${\mathcal B}$ will be countably generated, and $G$ will be metrisable (by the Urysohn metrisation theorem).

As $(G,{\mathcal B},\mu)$ is a measure space, it comes with the usual $L^p$ spaces $L^p(G) = L^p(\mu) = L^p(G,{\mathcal B},\mu)$ for $0 < p \leq \infty$.  We also define $L^\infty_c(G)$ to be the subspace of $L^\infty(G)$ consisting of functions that are compactly supported; this is a translation-invariant algebra, with the shift action\footnote{In many texts, the shift $T^h f$ is defined as $T^h f(x) = f(x+h)$ instead of $T^h f(x) = f(x-h)$.  The two conventions lead to an equivalent definition of the Gowers norms; this convention is slightly more compatible with the ergodic theory conventions, and lead to the pleasant identity $T^h 1_E = 1_{T^h E} := 1_{E+h}$ for indicator functions $1_E$.}  $T^h: L^\infty_c(G) \to L^\infty_c(G)$ defined for $h \in G$ by the formula $T^h f(x) := f(x-h)$.

\begin{definition}[Gowers norms on a LCA group]\cite{gowers-4aps,gowers-longaps}  Let $G=(G,+,{\mathcal B},\mu)$ be a locally compact abelian group with a non-trivial Haar measure $\mu$, and let $f \in L^\infty_c(G)$ be a function.  We define the \emph{Gowers uniformity norms} $\|f\|_{U^k(G)} = \|f\|_{U^k(\mu)} = \|f\|_{U^k(G,{\mathcal B},\mu)}$ recursively for $k=1,2,\ldots$ by the formula
\begin{equation}\label{f1}
 \|f\|_{U^1(G)} := \left|\int_G f\ d\mu\right|
\end{equation}
and
\begin{equation}\label{fk}
\|f\|_{U^{k+1}(G)} := \left(\int_G \|(T^h f) \overline{f} \|_{U^k(G)}^{2^k}\ d\mu(h)\right)^{1/2^{k+1}},
\end{equation}
for $k \geq 1$.
\end{definition}

It is possible to show that the Gowers uniformity norms are seminorms for any $k \geq 1$; see e.g. \cite{gowers-longaps}, \cite{green-tao-longaps}, or \cite{tao-vu}.  (In these references, the seminorm property is only established when the group $G$ is finite or cyclic, but the proof easily extends to the general LCA case.)  When $k=2$, we see from \eqref{f1}, \eqref{fk} that we have
$$ \|f\|_{U^2(G)} = \| \tilde f * f \|_{L^2(G)}^{1/2}$$
where $\tilde f(x) := \overline{f(-x)}$ and $f*g(x) := \int_G f(y) g(x-y)\ d\mu(y)$.  From Plancherel's theorem we thus have
\begin{equation}\label{fug2}
\|f\|_{U^2(G)} = \| \hat f \|_{L^4(\hat G)}
\end{equation}
where $\hat G = (\hat G, \hat{\mathcal B}, \hat \mu)$ is the Pontryagin dual of $G$, and $\hat f$ is the Fourier transform.  Because of this, the $U^2(G)$ seminorm is in fact a norm on $G$, and by induction we see that the $U^k(G)$ seminorms are also norms for all $k \geq 2$.

The norm can also be written more explicitly as
\begin{equation}\label{fug}
 \|f\|_{U^k(G)} := \left( \int_{G^{k+1}} \prod_{\omega \in \{0,1\}^k} {\mathcal C}^{|\omega|} f(x+h_1 \omega_1 + \ldots + h_k \omega_k)\ d\mu(x) d\mu(h_1) \ldots d\mu(h_k) \right)^{1/2^k}
 \end{equation}
where $\omega = (\omega_1,\ldots,\omega_k)$, $|\omega| := \omega_1 + \ldots + \omega_k$, and ${\mathcal C}: z \mapsto \overline{z}$ is the complex conjugation map.

From Young's inequality (see Proposition \ref{young} below, using the exponents \eqref{pqk}) and induction we obtain the inequality
\begin{equation}\label{ukg2}
\|f\|_{U^k(G)} \leq \|f\|_{L^{p_k}(G)}
\end{equation}
for all $k \geq 1$ and $f \in L^\infty_c(G)$, where $p_k$ is the \emph{critical exponent}
\begin{equation}\label{pkdef}
p_k := \frac{2^k}{k+1}
\end{equation}
associated to $k$, thus for instance
$$ p_1 = 1; \quad p_2 = 4/3; \quad p_3 = 2; \quad p_4 = 16/5; \ldots$$
This exponent is natural from the perspective of dimensional analysis, and in particular with respect to the operation of dilating the Haar measure $\mu$ to $c\mu$ for some scalar $c>0$.  Indeed, a short computation shows that
$$ \|f\|_{U^k(G, {\mathcal B}, c\mu)} = c^{1/p_k} \|f\|_{U^k(G,{\mathcal B},\mu)}.$$
As a consequence, as long as the Haar measure is not fixed, one cannot replace the $p_k$ exponent in \eqref{ukg2} by any other exponent.
On the other hand, in the compact case, the Haar measure is normalised so that $\mu(G)=1$, and then \eqref{ukg2} also holds for higher exponents than $p_k$ by H\"older's inequality, and in particular
\begin{equation}\label{ukg}
\|f\|_{U^k(G)} \leq \|f\|_{L^{\infty}(G)}.
\end{equation}

From \eqref{ukg2}, we may extend the $U^k(G)$ seminorm from $L^\infty_c(G)$ to all of $L^{p_k}(G)$ by continuity.  As the Fourier transform is an injection from $L^{4/3}(G)$ to $L^4(\hat G)$, we see that the $U^2(G)$ norm remains a norm on $L^{4/3}(G)$, and by induction the $U^k(G)$ norm remains a norm on $L^{p_k}(G)$ for $k \geq 2$.  By continuity, we see that the formulae \eqref{f1}, \eqref{fk} continue to be valid for $f \in L^{p_k}(G)$ (in particular, the integrands are absolutely integrable).

It is also convenient in applications to define the Gowers norms over the discrete interval $[N] := \{1,\ldots,N\}$.  Given a function $f: [N] \to\C$ and a positive integer $k$, we define the Gowers norm $\|f\|_{U^k([N])}$ of $f$ on $[N]$ by the formula
\begin{equation}\label{util}
 \|f\|_{U^k([N])} := \| \tilde f \|_{U^k(\Z/\tilde N\Z)} / \| 1_{[N]} \|_{U^k(\Z/\tilde N\Z)}
\end{equation}
where $\tilde N$ is any integer with $\tilde N > 2^k N$, $[N]$ is embedded into the cyclic group $\Z/\tilde N\Z$ in the obvious manner, and $\tilde f: \Z/\tilde N\Z \to \C$ is the extension of $f$ by zero outside of $[N]$.  It is easy to see that this definition does not depend on the precise choice of $\tilde N$.  Clearly we still have the analogue of \eqref{ukg}:
\begin{equation}\label{ukgn}
\|f\|_{U^k([N])} \leq \|f\|_{L^{\infty}([N])}.
\end{equation}
The situation with \eqref{ukg2} however appears to be more complicated - there seems to be a loss of a multiplicative constant in the $[N]$ setting - and we will not study it here.

Next, we turn to the \emph{Gowers-Host-Kra seminorms} arising from an ergodic measure-preserving system $(X, {\mathcal X}, \mu, T)$, by which we mean a probability space $(X, {\mathcal X}, \mu)$ together with an invertible, bimeasurable, measure-preserving shift $T: X \to X$ which acts ergodically, so that the invariant space $L^2(X, {\mathcal X}, \mu)^T$ of $L^2(X, {\mathcal X}, \mu)$ consists only of the constants (up to almost everywhere equivalence).  We abbreviate ``ergodic measure-preserving system'' as ``ergodic system'' for short.

\begin{definition}[Gowers-Host-Kra seminorms]\label{semidef}\cite{host-kra} Let $(X, {\mathcal X}, \mu, T)$ be an ergodic system, and let $f \in L^\infty(X, {\mathcal X}, \mu)$.  We define the  \emph{Gowers-Host-Kra seminorms} $\|f\|_{U^k(X)} = \|f\|_{U^k(X,{\mathcal X},\mu,T)}$ recursively for $k=1,2,\ldots$ by the formula
\begin{equation}\label{us1}
 \|f\|_{U^1(X)} := \left|\int_X f(x)\ d\mu(x)\right|
\end{equation}
and
\begin{equation}\label{usx}
 \|f\|_{U^{k+1}(X)} := \lim_{H \to \infty} \left(\E_{h \in [H]} \|(T^h f) \overline{f} \|_{U^k(X)}^{2^k}\right)^{1/2^{k+1}}
\end{equation}
for $k \geq 1$, where $T^h f := f \circ T^{-h}$.
\end{definition}

The existence of these seminorms, and the fact that these are indeed seminorms, is established in \cite{host-kra}.  The definition of the norms used there appears to be somewhat different from that presented here, but the two definitions are equivalent, as can be seen from an application of the ergodic theorem; see e.g. \cite[Appendix A]{bergelson-tao-ziegler}.

\begin{remark} It is possible to define the Gowers-Host-Kra seminorms on non-ergodic spaces as well, but in such cases the definition of the $U^1(X)$ norm must be replaced with
$$ \|f\|_{U^1(X)} := \lim_{H \to \infty} \left(\E_{h \in [H]} \int_X T^h f(x) \overline{f}(x)\ d\mu(x)\right)^{1/2}.$$
One can relate the non-ergodic norms with the ergodic norms via ergodic decomposition as follows.  Let $X^T := (X, {\mathcal X}^T, \mu^T, \id)$ be the invariant factor of $X$, thus ${\mathcal X}^T := \{ E \in {\mathcal X}: TE = E \}$ and $\mu^T$ is the restriction of $\mu$ to ${\mathcal X}$, and one has an ergodic decomposition\footnote{Strictly speaking, to construct this decomposition we need to assume that the system is \emph{standard Borel}, in that $X$ has the structure of a Polish space with ${\mathcal X}$ as its Borel $\sigma$-algebra.  In practice, we will be able to reduce to the standard Borel case in most applications, so this technicality will not concern us.  See \cite{furst-book} for further discussion.}
$\mu = \int_{X^T} \mu_y\ d\mu^T(y)$ for some $T$-invariant ergodic measures $\mu_y$ for $\mu^T$-almost every $y \in X^T$ that are depending in a measurable fashion on the parameter $y \in X^T$.  Then one can show that
$$ \|f\|_{U^k(X)} = \left(\int_{X^T} \|f\|_{U^k(X,{\mathcal X},\mu_y)}^{2^k}\ d\mu^T(y)\right)^{1/2^k}.$$
From this inequality, it is possible to extend some of the results given here for ergodic systems to the non-ergodic setting, although the statements of the hypotheses and conclusions become significantly messier.  Because of this, we will not discuss the non-ergodic case here.
\end{remark}

\begin{example}\label{uni-kra}
The Gowers uniformity norm $\|f\|_{U^k(\Z/N\Z)}$ on a cyclic group $\Z/N\Z$ can be viewed as a special case of the Gowers-Host-Kra seminorms, in which $X$ is equal to $\Z/N\Z$ with the discrete $\sigma$-algebra ${\mathcal X} := 2^{\Z/N\Z}$, the normalised counting measure $\mu(E) := |E|/N$, and the standard shift $T x := x+1$ (which is clearly ergodic); the equivalence of the two (semi)norms can be verified by an easy induction.  More generally, if $G$ is a compact abelian group, with an ergodic shift $x \mapsto x+\alpha$ and normalised Haar measure, then an easy induction shows that the Gowers norm $U^k(G)$ and the Gowers-Host-Kra seminorm $U^k(G)$ coincide.
\end{example}

The analogue of Young's inequality for ergodic systems (see Proposition \ref{young-ergodic}) gives the analogues of \eqref{ukg2}, \eqref{ukg}, namely that
\begin{equation}\label{ukp}
\|f\|_{U^k(X)} \leq \|f\|_{L^{p_k}(X)}
\end{equation}
and thus
\begin{equation}\label{ukp2}
 \|f\|_{U^k(X)} \leq \|f\|_{L^\infty(X)}
\end{equation}
for all $f \in L^\infty(X)$.  Because of this, these seminorms can be extended by continuity to $L^{p_k}(X)$ much as before.  In contrast to the group case, though, these seminorms can vanish.  Indeed, this occurs precisely when $f$ is orthogonal to the \emph{characteristic factor} $Z_{k-1}$ of the $U^k$ norms, which is an important object in the theory of multiple recurrence; see \cite{host-kra}.

\begin{remark}  The (semi)norms $U^k(G)$, $U^k([N])$, $U^k(X)$ are so similar that it is tempting to create a very general definition of Gowers-type norms that encompasses all three concepts at once.  It seems that the machinery of nilspaces (see \cite{host-kra-parallel}, \cite{cs}; the concept is a variant of the notion of a \emph{cubic complex} from topology) would be particularly suitable for this purpose, as would a reformulation in terms of nonstandard analysis (or by closely related devices, such as ultraproducts).  However, we will not attempt to perform such a unification here, thus creating three parallel (but closely analogous) streams of results instead.  One reason for this is that there are some non-trivial technical differences between the \emph{proofs} of the results in the three categories, particularly with regards to the $U^k([N])$ category, where the failure of $[N]$ to be perfectly closed under addition leads to some complications when one attempts to adapt arguments that were successful in the other two categories.
\end{remark}

\subsection{Near-maximal values of the Gowers norms and Gowers-Host-Kra seminorms}

The first main results of this paper address the question of when the inequalities \eqref{ukg}, \eqref{ukg2}, \eqref{ukgn}, \eqref{ukp}, \eqref{ukp2} are satisfied with equality or near-equality.  For the $L^\infty$-based inequalities \eqref{ukg}, \eqref{ukgn}, \eqref{ukp}, the answer to this question is given in terms of the notion of a \emph{polynomial}.   We first define this concept on groups:

\begin{definition}[Polynomials on groups]  Let $k \geq -1$ be an integer. A measurable map $P: G \to H$ between two LCA groups $G, H$ is called a \emph{polynomial map of degree $\leq k$} if one has
$$ \Delta_{h_1} \ldots \Delta_{h_{k+1}} P(x) = 0$$
for almost every $x,h_1,\ldots,h_{k+1} \in G$, where $\Delta_h P(x) := P(x-h)-P(x)$.
\end{definition}

\begin{example} The only polynomial map of degree $\leq -1$ is the zero map.  The only polynomials of degree $\leq 0$ are the constant maps.  The only polynomials of degree $\leq 1$ are affine homomorphisms, i.e. sums of constants and group homomorphisms.
\end{example}

An easy induction (using the strong continuity properties of the translation action $T^h: P \mapsto P(\cdot-h)$) shows that a measurable polynomial is equal almost everywhere to a continuous polynomial.  Of course, in the case of discrete groups $G$, and in particular finite groups, all polynomials are automatically measurable and continuous.

In the case of the integers $G=\Z$, another easy induction shows that a map $P: \Z \to H$ is a polynomial of degree $\leq k$ if and only if it takes the form
$$ f(n) = \sum_{i=0}^k \binom{n}{i} c_i$$
for some coefficients $c_0,\ldots,c_k \in H$, where
$$ \binom{n}{i} := \frac{n(n-1)\ldots(n-i+1)}{i!}$$
are the (generalised) binomial coefficients.  In the case when the target group $H$ is \emph{divisible}, which means that for every $x \in H$ and positive integer $m$ there exists $y \in H$ such that $my=x$, we can rewrite $f$ in the form
$$ f(n) = \sum_{i=0}^k n^i c'_i$$
for some other coefficients $c'_0,\ldots,c'_k \in H$.  We note that any connected compact abelian Lie group is divisible, and in particular the unit circle $\R/\Z$ is divisible.

In the case when $G$ is a Euclidean space $\R^d$, and $H$ is the unit circle $\R/\Z$, then it is easy to see (by approximating the Euclidean space $\R^d$ by a lattice $\eps \Z^d$ and then taking limits as $\eps \to 0$) that a map $P: \R^d \to \R/\Z$ is polynomial of degree $\leq k$ if and only if it takes the form
$$ f(x_1,\ldots,x_d) := \sum_{i_1,\ldots,i_d \ge 0: i_1+\ldots+i_d \leq k} c_{i_1,\ldots,i_d} x_1^{i_1} \ldots x_d^{i_d} \mod 1$$
for some coefficients $c_{i_1,\ldots,i_d} \in \R$.

We can similarly define the notion of a polynomial on ergodic systems:

\begin{definition}[Polynomials on systems]  Let $k \geq -1$ be an integer. A measurable map $f: X \to H$ between an ergodic system $X = (X,{\mathcal X},\mu,T)$ and a LCA group $H$ is called a \emph{polynomial map of degree $\leq k$} if one has
$$ \Delta_{h_1} \ldots \Delta_{h_{k+1}} f(x) = 0$$
for $\mu$-almost every $x \in X$ and every $h_1,\ldots,h_{k+1} \in \Z$, where $\Delta_h := T^h - 1$.
\end{definition}

\begin{example} Let $f \in L^\infty(X)$ be a non-trivial eigenfunction of $X$, thus $f$ is not identically zero and $Tf = \lambda f$ for some complex number $\lambda$ (which one can show to necessarily have magnitude $1$).  Then one can show that $f = c \cdot e(P)$ for some complex constant $c$ and some polynomial $P: X \to \R/\Z$ of degree $\leq 1$, where $e(x) :=e^{2\pi ix}$.  Conversely, every function of this form is an eigenfunction of $X$.

Functions of the form $c \cdot e(P)$, where $P:X \to \R/\Z$ is of degree $\leq k$, are known as \emph{generalised eigenfunctions of order $\leq k$}.  For instance, if we take the skew shift system $X := (\R/\Z)^2$ with the standard measure and the shift $T(x,y) := (x+\alpha,y+x)$ for a fixed irrational $\alpha$, then the function $f(x,y) := e(y)$ is a generalised eigenfunction of order $\leq 2$; note that $\Delta_h f(x,y) = e(-hx - \binom{-h}{2}\alpha)$ is an ordinary eigenfunction for every $h \in \Z$.
\end{example}

The relevance of polynomials to the Gowers norms can be seen as follows.  If $k \geq 1$, $G$ is a LCA group, and $P: G \to \R/\Z$ is a polynomial of degree $\leq k-1$, then we have the invariance
\begin{equation}\label{pepg}
 \| f \cdot e(P) \|_{U^k(G)} = \|f\|_{U^k(G)}
\end{equation}
for every $f \in L^{p_k}(G)$.  Similarly we have
\begin{equation}\label{pepn}
 \| f \cdot e(P) \|_{U^k([N])} = \|f\|_{U^k([N])}
\end{equation}
for every $f \in L^{p_k}([N])$ and polynomial $P: \Z \to \R/\Z$ of degree $\leq k-1$, and
\begin{equation}\label{pepx}
 \| f \cdot e(P) \|_{U^k(X)} = \|f\|_{U^k(X)}
\end{equation}
for every ergodic system $X$, every $f \in L^{p_k}(X)$, and every polynomial $P: X \to \R/\Z$ of degree $\leq k-1$.  In particular, we see that equality holds in \eqref{ukg}, \eqref{ukgn}, or \eqref{ukp} when $f$ takes the form $f = c \cdot e(P)$ for a polynomial $P$ (of the suitable type) of degree $\leq k-1$.

Our first main results are that these are in fact the \emph{only} cases in which equality in \eqref{ukg}, \eqref{ukgn}, \eqref{ukp} occurs, and they also control all the cases in which \emph{near-equality} occurs.  To formulate this properly, it is convenient to introduce the following notation.  Given an asymptotic parameter $\eps \geq 0$ and some additional quantities $a_1,\ldots,a_m,A$, we use $o_{\eps \to 0; a_1,\ldots,a_m}(A)$ to denote any expression bounded in magnitude by $c_{a_1,\ldots,a_m}(\eps) A$, where $c_{a_1,\ldots,a_m}(\eps)$ is an expression depending only on $a_1,\ldots,a_m,\eps$ that goes to zero as $\eps \to 0$ for fixed $a_1,\ldots,a_m$.  We will use the expression $o_{\eps \to 0;k}(A)$ particularly often, and will abbreviate this expression as $o(A)$.  In most situations in this paper, $A$ will simply be equal to $1$.

\begin{theorem}[$L^\infty$ near-extremisers on compact abelian groups]\label{infty-group}  Let $k \geq 1$ be an integer, let $G$ be a compact abelian group, and let $f \in L^\infty(G)$ be such that $\|f\|_{L^\infty(G)} \leq 1$.  Let $\eps \geq 0$.
\begin{enumerate}
\item (Extremisers)  One has $\|f\|_{U^k(G)} = 1$ if and only if $f = e(P)$ almost everywhere for some polynomial $P: G \to \R/\Z$ of degree $\leq k-1$.
\item (Near-extremisers) If $\|f\|_{U^k(G)} \geq 1-\eps$, then there exists a polynomial $P: G \to \R/\Z$ of degree $\leq k-1$ such that $\|f-e(P)\|_{L^1(G)} = o(1)$.
\end{enumerate}
\end{theorem}

\begin{theorem}[$L^\infty$ near-extremisers on intervals]\label{infty-N}  Let $k \geq 1$ be an integer, and let $N \geq 1$ be a sufficiently large integer depending on $k$.  Let $f \in L^\infty([N])$ be such that $\|f\|_{L^\infty([N])} \leq 1$.  Let $\eps \geq 0$.
\begin{enumerate}
\item (Extremisers)  One has $\|f\|_{U^k([N])} = 1$ if and only if $f = e(P)$ for some polynomial $P: \Z \to \R/\Z$ of degree $\leq k-1$.
\item (Near-extremisers) If $\|f\|_{U^k([N])} \geq 1-\eps$, then there exists a polynomial $P: \Z \to \R/\Z$ of degree $\leq k-1$ such that $\|f-e(P)\|_{L^1([N])} = o(1)$.
\end{enumerate}
\end{theorem}

\begin{theorem}[$L^\infty$ near-extremisers on ergodic systems]\label{infty-ergodic}  Let $k \geq 1$ be an integer, let $X$ be an ergodic measure-preserving system, and let $f \in L^\infty(X)$ be such that $\|f\|_{L^\infty(X)} \leq 1$.  Let $\eps \geq 0$.
\begin{enumerate}
\item (Extremisers)  One has $\|f\|_{U^k(X)} = 1$ if and only if $f = e(P)$ almost everywhere for some polynomial $P: X \to \R/\Z$ of degree $\leq k-1$.
\item (Near-extremisers) If $\|f\|_{U^k(X)} \geq 1-\eps$, then there exists a polynomial $P: X \to \R/\Z$ of degree $\leq k-1$ such that $\|f-e(P)\|_{L^1(X)} = o(1)$.
\end{enumerate}
\end{theorem}

We will prove these theorems in Sections \ref{infty-group-sec}, \ref{infty-n-sec}, \ref{infty-system-sec} respectively.  We remark that Theorem \ref{infty-group}, which can be interpreted as an assertion that the property of being a polynomial is locally testable, was essentially established in the case when $G$ is a vector space over a finite field in \cite{akklr} (see also \cite{blr} for the $k=2$ case, or \cite{tao-focs}, \cite{tao-ziegler} for a more explicit formulation of this result).

As one might expect, the proofs of the three results are very similar to each other, and proceed by an induction on $k$; they are the easiest of all the results in this paper to prove.  The main difficulty is a ``cohomological'' one, namely to show that a certain ``$2$-cocycle'' arising from applying the induction hypothesis to ``derivatives'' $(T^h f) \overline{f}$ of $f$ is in fact a ``$2$-coboundary''.  However, when $\eps$ is small enough, this $2$-cocycle takes on small values, and one can obtain this $2$-coboundary property by a routine averaging argument.  (The situation is more delicate on $[N]$, as the domain is no longer shift-invariant, but a more sophisticated version of this argument still applies.)  We will not actually use advanced cohomological tools in our arguments, though, and the reader may ignore the references to cohomological notation in this paper if desired.

\begin{remark} Our arguments give an effective bound on the $o(1)$ decay rates, which are of polynomial nature on $\eps$.  However, we will not attempt to optimise or make explicit these rates here.
\end{remark}

Next, we consider extreme or near-extreme cases of the critical inequalities \eqref{ukg2}, \eqref{ukp2}, which compare the Gowers norms $U^k$ to their Lebesgue counterparts $L^{p_k}$.  From the polynomial phase invariance \eqref{pepg}, \eqref{pepx} we expect polynomial phases $e(P)$ to continue to play a prominent role.  However, due to the critical nature of these inequalities, another object now also comes into play, namely the \emph{cosets}.

\begin{definition}[Cosets in a LCA group]  Let $G=(G,+,{\mathcal B},\mu)$ be an LCA group.  A \emph{coset} in $G$ is any set of the form $H = x_0 + H_0$, where $x_0 \in G$ and $H_0 \leq G$ is a compact open subgroup of $G$ (which implies in particular that $0 < \mu(H_0) < \infty$).  We define the \emph{normalisation} of the coset $H$ to be the compact abelian group $H_0$ with the normalised Haar measure $\frac{1}{\mu(H_0)} \mu \downharpoonright_{H_0}$.
\end{definition}

\begin{definition}[Cosets in an ergodic system]\label{cosets}  Let $X=(X,{\mathcal X},\mu,T)$ be an ergodic system.  A \emph{coset} in $X$ is any measurable set $H$ in $X$ with the property that there exists an integer $m \geq 1$ such that $T^m H = H$ up to $\mu$-null sets, and that the sets $H, TH, \ldots, T^{m-1} H$ partition $X$ up to $\mu$-null sets.  (In particular, this forces $\mu(H) = 1/m$.)  We call $m$ the \emph{index} of $H$.  We define the \emph{normalisation} of the coset $H$ to be the measure-preserving system $H = (H, {\mathcal X}\downharpoonright_H, \frac{1}{\mu(H)} \mu\downharpoonright_H, T^m)$.
\end{definition}

\begin{example}  If $G=X=\Z/N\Z$ with the usual shift, then the above two notions of coset coincide with each other, and with the familiar notion of a coset from group theory.    More generally, a coset in an ergodic system $X$ is the same concept as a group-theoretic coset of the \emph{Kronecker factor} $Z_1$ of $X$; see Remark \ref{coset}.
\end{example}

\begin{example} An ergodic system has no nontrivial cosets if and only if is \emph{totally ergodic}, in that $T^h$ is ergodic for every non-zero $h$.
\end{example}

We observe that the normalisation of a coset $H$ in an ergodic measure-preserving system $X$ is again ergodic.  Indeed, if $f \in L^2(H)$ is $T^m$-invariant, then $f + Tf + \ldots + T^{m-1} f$ is $T$-invariant in $L^2(G)$, where we extend $f$ by zero to all of $G$, and the claim follows.

The relevance of cosets arises from the following observation.  If $G$ is an LCA group with a coset $H = x_0+H_0$ and normalisation $H_0$, and $\tilde f \in L^{p_k}(H_0)$, then straightforward induction shows that the function $f \in L^{p_k}(G)$ defined by $f(x) := \mu(H)^{-1/p_k} 1_H(x) \tilde f(x-x_0)$ obeys the scaling relationships
$$ \|f\|_{U^k(G)} = \|\tilde f\|_{U^k(H)}$$
and
$$ \|f\|_{L^{p_k}(G)} = \|\tilde f\|_{L^{p_k}(H)}$$
where of course we use the normalised measure to compute the norms on the right-hand side.
Similarly, if $X$ is an ergodic measure-preserving system with a coset $H$, and $\tilde f \in L^{p_k}(H)$, then another straightforward induction shows that the function $f \in L^{p_k}(G)$ defined by $f(x) := \mu(H)^{-1/p_k} 1_H(x) \tilde f(x)$ obeys the scaling relationships
$$ \|f\|_{U^k(X)} = \|\tilde f\|_{U^k(H)}$$
and
$$ \|f\|_{L^{p_k}(X)} = \|\tilde f\|_{L^{p_k}(H)}.$$

Because of this, we see that we have a more general class of extremisers for \eqref{ukg2}, \eqref{ukp2}.  Indeed, if $H=x_0+H_0$ is a coset of a LCA group $G$ and $P: H_0 \to \R/\Z$ is a polynomial of degree $\leq k-1$, then from the above discussion we see that the function $f(x) := c 1_H(x) e( P(x-x_0))$, for any constant $c \in \C$, obeys \eqref{ukg2} with equality.  Similarly, if $H$ is a coset of an ergodic system $X$, and $P: H \to \R/\Z$ is a polynomial of degree $\leq k-1$, then the function $f(x) := c 1_H(x) e(P(x))$ for any constant $c \in \C$ obeys \eqref{ukp2} with equality.

Our next two main results assert that these are the essentially only means to create extremals or near-extremals for \eqref{ukg2}, \eqref{ukp2}, once one avoids the degenerate case $k=1$ (for which there are clearly plenty of such near-extremals):

\begin{theorem}[$L^{p_k}$ near-extremisers on LCA groups]\label{p-group}  Let $k \geq 2$ be an integer, let $G$ be a locally compact abelian group, and let $f \in L^{p_k}(G)$ be such that $\|f\|_{L^{p_k}(G)} \leq 1$.  Let $\eps \geq 0$.
\begin{enumerate}
\item (Extremisers)  One has $\|f\|_{U^k(G)} = 1$ if and only if
$$f(x) = \mu(H)^{-1/p_k} 1_H(x) e(P(x-x_0))$$
almost everywhere for some coset $H = x_0+H_0$ and polynomial $P: H_0 \to \R/\Z$ of degree $\leq k-1$.
\item (Near-extremisers) If $\|f\|_{U^k(G)} \geq 1-\eps$, then there exists a coset $H = x_0+H_0$ and a polynomial $P: H_0 \to \R/\Z$ of degree $\leq k-1$ such that
$$\|f-\mu(H)^{-1/p_k} 1_H e(P(\cdot-x_0)) \|_{L^{p_k}(G)} = o(1).$$
\end{enumerate}
\end{theorem}

\begin{theorem}[$L^{p_k}$ near-extremisers on ergodic systems]\label{p-system}  Let $k \geq 2$ be an integer, let $X$ be an ergodic system, and let $f \in L^{p_k}(X)$ be such that $\|f\|_{L^{p_k}(X)} \leq 1$.  Let $\eps \geq 0$.
\begin{enumerate}
\item (Extremisers)  One has $\|f\|_{U^k(X)} = 1$ if and only if $f = \mu(H)^{-1/p_k} 1_H e(P)$ almost everywhere for some coset $H$ and polynomial $P: H \to \R/\Z$ of degree $\leq k-1$.
\item (Near-extremisers) If $\|f\|_{U^k(X)} \geq 1-\eps$, then there exists a coset $H$ and a polynomial $P: H \to \R/\Z$ of degree $\leq k-1$ such that
$$\|f-\mu(H)^{-1/p_k} 1_H e(P) \|_{L^{p_k}(X)} = o(1).$$
\end{enumerate}
\end{theorem}

We prove these theorems in Sections \ref{p-group-sec},
\ref{p-system-sec} respectively.  The main idea is to use a
classification of the near-extremisers of Young's inequality, due to
Fournier~\cite{fournier}, to reduce matters to the $L^\infty$ theory discussed earlier.  (Indeed, the paper \cite{fournier} already implicitly contains the $k=2$ version of Theorem \ref{p-group}.)  To prove Theorem \ref{p-system} we also exploit the theory of the \emph{Kronecker factor} of an ergodic system.


\subsection{The Euclidean case}

Theorem \ref{p-group} has a corollary, which was essentially observed
by Fournier~\cite{fournier} (and by Russo~\cite{russo} when $k=2$): if a locally compact group $G$ has no compact open subgroups, then one has
$$ \|f\|_{U^k(G)} \leq c_{k,G} \|f\|_{L^{p_k}(G)}$$
for some constant $c_{k,G} < 1$ that is bounded away from $1$ uniformly in $G$.  In the case that $G$ is a Euclidean space $\R^n$ (with the usual Lebesgue measure, of course), we can compute the optimal value of $c_{k,G}$ precisely:

\begin{theorem}[Sharp critical inequality for $U^k(\R^n)$]\label{syg}  Let $k,n \geq 1$.  Then for any $f \in L^{p_k}(\R^n)$, one has
\begin{equation}\label{ud-sharp}
 \|f\|_{U^k(\R^n)} \leq C_k^n  \|f\|_{L^{p_k}(\R^n)}
 \end{equation}
where $C_k$ is the constant
\begin{equation}\label{C_d}
C_k := 2^{k/2^k} / (k+1)^{(k+1)/2^{k+1}}.
\end{equation}
This constant is best possible, and when $k \geq 2$, equality is attained if and only if $f$ takes the form
\begin{equation}\label{fphi}
 f(x) = c e^{-(x-x_0) \cdot M (x-x_0)} e( \phi(x) )
\end{equation}
for some $c \in \C$, $x_0 \in \R^n$, a positive-definite $n \times n$ matrix $M$, and a polynomial $\phi: \R^n \to \R/\Z$ of degree $\leq k-1$.
\end{theorem}

The first few values of $C_k$ are
\begin{align*}
C_1 &= \frac{2^{1/2}}{2^{2/4}} = 1\\
C_2 &= \frac{2^{2/4}}{3^{3/8}} \approx 0.9367\\
C_3 &= \frac{2^{3/8}}{4^{4/16}} = 2^{-1/8} \approx 0.9170\\
C_4 &= \frac{2^{4/16}}{5^{5/32}} \approx 0.9248.
\end{align*}

We prove Theorem \ref{syg} in Section \ref{syg-sec}.  Our main tool is
the sharp Young inequality due to Beckner \cite{beckner} and Brascamp and Lieb \cite{brascamp}, and its converse.  It is quite likely that one could also establish near-extremiser results analogous to those appearing previously in this introduction, but we will not attempt to do so here.  (A starting point would be to first establish near-extremiser results for the sharp Young inequality, which could perhaps be deduced using the machinery from \cite{barthe}, \cite{barthe2}.)

\subsection{Threshold behaviour for the $U^3$ norm}

It is natural to ask how small the quantity $1-\eps$ appearing in results such as Theorem \ref{p-group} or Theorem \ref{p-system} can be.  For $k=1,2$ we can lower $1-\eps$ all the way to zero:

\begin{theorem}[Inverse theorem for $k=1,2$]\label{th12} Let $k=1,2$ and let $\eps > 0$.
\begin{itemize}
\item If $G$ is a compact abelian group (with normalised Haar measure) and $f \in L^\infty(G)$ is such that $\|f\|_{L^\infty(G)} \leq 1$ and $\|f\|_{U^k(G)} \geq\eps$, then there exists a polynomial $P: G \to \R/\Z$ of degree at most $k-1$ such that $|\langle f, e(P) \rangle_{L^2(G)}| \geq c(k,\eps)$, where $c(k,\eps) > 0$ is a quantity that depends only on $k$ and $\eps$.
\item If $N$ is a positive integer and $f \in L^\infty([N])$ is such that $\|f\|_{L^\infty([N])} \leq 1$ and $\|f\|_{U^k([N])} \geq\eps$, then there exists a polynomial $P: [N] \to \R/\Z$ of degree at most $k-1$ such that $|\langle f, e(P) \rangle_{L^2([N])}| \geq c(k,\eps)$, where $c(k,\eps) > 0$ is a quantity that depends only on $k$ and $\eps$.
\item  If $X$ is an ergodic system and $f \in L^{p_k}(X)$ is such that $\|f\|_{L^{p_k}(G)} \leq 1$ and $\|f\|_{U^k(G)} > 0$, then there exists a polynomial $P: X \to \R/\Z$ of degree at most $k-1$ such that $|\langle f, e(P) \rangle_{L^2(X)}| > 0$.
\end{itemize}
\end{theorem}

\begin{proof}  The case $k=1$ is easily verified by inspection, so we establish the $k=2$ case only.  In the group case, we can use \eqref{fug2} and Plancherel's theorem to conclude that
$$ \| \hat f \|_{\ell^2(\hat G)} \leq 1$$
and
$$ \| \hat f \|_{\ell^4(\hat G)} \geq \eps$$
which implies that there exists $\xi \in \hat G$ such that $|\hat f(\xi)| \geq \eps^2$, and the claim follows.

The interval case $[N]$ follows easily from the group case with $G = \Z/\tilde N\Z$ for a suitable $\tilde N$.  In the ergodic system case, it is possible to perform a similar argument using the spectral decomposition of the shift $T$, but we will instead use the (closely related) theory of the \emph{Kronecker factor} from Host and Kra \cite{host-kra}.  It suffices to show that if $f \in L^{4/3}(X)$ is orthogonal to all eigenfunctions of $X$, then it has a $U^2(G)$ norm of zero.  Let ${\mathcal Z}_1$ be the sub-$\sigma$-algebra of ${\mathcal X}$ generated by the eigenfunctions of $X$ (or equivalently, by the pure point spectrum of $T$), then $f$ is orthogonal to ${\mathcal Z}_1$.  The claim then follows\footnote{Strictly speaking, the lemma cited only applies when $f$ lies in $L^\infty(X)$, but it can be easily extended to $f$ in $L^{4/3}(X)$ by a limiting argument using \eqref{ukp}.} from \cite[Lemma 4.3]{host-kra}.
\end{proof}

However, the situation changes when $k=3$: the parameter $\eps$ in
Theorem \ref{th12} can no longer be lowered all the way to zero.  In
the group case (and specifically, for cyclic groups $G=\Z/N\Z$ with
$N$ large) and in the case of the interval $[N]$, this was observed by
Gowers \cite{gowers-4aps}; the analogous observation for ergodic
systems is implicit in the work of Furstenberg and Weiss \cite{furst-legacy}, \cite{fw} (in the closely related context of \emph{multiple recurrence}), and more explicitly in the work of Host and Kra \cite{host-kra}.  On the other hand, Theorem \ref{th12} for arbitrarily small $\eps$ can be recovered in the $k=3$ case for some groups $G$, such as vector spaces $G = \mathbb{F}_p^n$ over a fixed characteristic $p$; see \cite{gt:inverse-u3}, \cite{sam}.

Our final main results are to compute the precise threshold for which Theorem \ref{th12} continues to hold in the $k=3$ case, provided that the group $G$ is cyclic, or that the system $X$ is totally ergodic.  More precisely, we have

\begin{theorem}[Sharp $U^3$ threshold for systems]\label{u3-system}  If $X$ is a totally ergodic system (thus $T^h$ is ergodic for every non-zero $h$) and $f \in L^2(X)$ is such that $\|f\|_{L^2(X)} \leq 1$ and $\|f\|_{U^3(X)} > 2^{-1/8}$, then there exists a polynomial $P: X \to \R/\Z$ of degree at most $2$ such that $|\langle f, e(P) \rangle_{L^2(X)}| > 0$.  Furthermore, the constant $2^{-1/8}$ is best possible.
\end{theorem}

\begin{theorem}[$U^3$ threshold for intervals]\label{u3-n}  If $N$ is an integer and $f \in L^\infty([N])$ is such that $\|f\|_{L^\infty([N])} \leq 1$ and $\|f\|_{U^3([N])} \geq 2^{-1/8}+\eta$ for some $\eta>0$, then there exists a polynomial $P: [N] \to \R/\Z$ of degree at most $2$ such that $|\langle f, e(P) \rangle_{L^2([N])}| > c(\eta)$ for some $c(\eta) > 0$.
\end{theorem}

\begin{theorem}[$U^3$ threshold for cyclic groups]\label{u3-cyclic}  If $N$ is an integer and $f \in L^\infty(\Z/N\Z)$ is such that $\|f\|_{L^\infty(\Z/N\Z)} \leq 1$ and $\|f\|_{U^3(\Z/N\Z)} \geq 2^{-1/8}+\eta$ for some $\eta>0$, then there exists a positive integer $1 \leq q \leq C(\eta)$ and a polynomial $P: \Z/qN\Z \to \R/\Z$ of degree at most $2$ such that $|\langle f, e(P) \rangle_{L^2(\Z/qN\Z)}| > c(\eta)$ for some $C(\eta), c(\eta) > 0$, where we lift $f$ from $\Z/N\Z$ up to $\Z/qN\Z$ in the obvious manner.
\end{theorem}

We prove these theorems in Sections \ref{u3-sys}, \ref{u3-n-sec}, \ref{u3-cyclic-sec} respectively.  The constant $2^{-1/8}$ is, not coincidentally, the constant $C_3$ that appears in Theorem \ref{syg}, and indeed we will use Theorem \ref{syg} as a key tool in establishing the above results.  We will also rely crucially on the inverse theorems of Host and Kra \cite{host-kra} (in the ergodic case) and of Green and Tao \cite{gt:inverse-u3} (in the cyclic group and interval cases), and on the structural theory of $2$-step nilmanifolds (as developed for instance in \cite{host-kra-2step}).  We also crucially take advantage of the fact that the critical exponent $p_2$ is equal to $2$, allowing us to exploit Plancherel's theorem at a key juncture to simplify the expressions being computed.  The relevance of nilsystems for Gowers-type norms was first explicitly noted in \cite{host-kra}, although nilsystems had also been implicitly linked to the closely related problem of multiple recurrence in \cite{conze}, \cite{fw}, \cite{hk01}.

The total ergodicity hypothesis in Theorem \ref{u3-system} is needed to rule out a certain type of dynamical system that is not obviously modeled by the Euclidean case from Theorem \ref{syg}, and it is possible that the $2^{-1/8}$ threshold may change if one no longer assumes total ergodicity; see Remark \ref{total}.  We will deduce Theorem \ref{u3-cyclic} from Theorem \ref{u3-n}.  The need to 
 lift up to a finite extension $\Z/qN\Z$ of $\Z/N\Z$ is technical, and is due to functions $f: \Z/N\Z \to \C$ such as
$$ f(n) := e( n^2 / qN )$$
for $n = 0,\ldots,N-1$, which has a large $U^3(\Z/N\Z)$ norm, but has low correlation with $e(P)$ for any quadratic $P: \Z/N\Z \to \R/\Z$ with period $N$, if $N$ is large compared with $q$.  It may be possible that the threshold $2^{-1/8}$ is high enough to exclude such examples, but we have not pursued this matter here.  We do not know if the constant $2^{-1/8}$ is best possible for Theorem \ref{u3-cyclic}; there is a technical difficulty when trying to embed the sharp counterexamples from an interval $[N]$ into a cyclic group $\Z/N\Z$ with no loss in constants.  We also do not know if the constant $2^{-1/8}$ is best possible for Theorem \ref{u3-n}, due to the issue of a possible gap between the $L^\infty([N])$ and $L^2([N])$ norms.

We do not know what the analogues of Theorems \ref{u3-system}, \ref{u3-n}, \ref{u3-cyclic} are for higher step.  The inverse theorem of Host-Kra and Green-Tao extend to higher step (see \cite{host-kra}, \cite{gtz} respectively), which essentially reduces matters to checking that the claimed assertions are true for nilsystems (in the ergodic case) or for nilsequences (in the interval and cyclic group cases). However, the structure of nilmanifolds of step $3$ or higher is significantly more complicated than in the $2$-step case, but more importantly, the critical exponent $p_k$ is no longer equal to $2$, and certain exact identities arising from the Plancherel theorem are no longer available to simplify the problem.

The authors are indebted to Tim Austin for posing these questions, and the anonymous referee for many useful comments and corrections.

\section{$L^\infty$ near-extremisers on groups}\label{infty-group-sec}

We now prove Theorem \ref{infty-group}, which is the easiest of the near-extremiser results.  We will just prove the claim for near-extremals with $\eps > 0$; the claim for extremals is simpler and can in any event be obtained from the $\eps > 0$ case by a limiting argument (or by adapting the proof).

We will need a simple lemma from \cite[Lemma C.1]{bergelson-tao-ziegler}:

\begin{lemma}[Separation lemma]\label{pcon}  Let $k \geq 1$ be an integer, and let $P: G \to \R/\Z$ be a polynomial of degree $\leq k$ such that $\| e( P ) - 1\|_{L^2(G)} < 2^{-k+1/2}$.  Then $P$ is constant.
\end{lemma}

\begin{proof}  When $k=1$ the claim is immediate from Fourier analysis and Pythagoras' theorem (since $e(P)$ is orthogonal to $1$ when $P$ is a non-constant linear polynomial), so suppose that $k>1$ and the claim has already been proven for $k-1$.  If
$$ \| e( P ) - 1\|_{L^2(G)} < 2^{-k+1/2},$$
then for any $h \in G$ we have
$$ \| e( T^h P ) - 1\|_{L^2(G)} < 2^{-k+1/2},$$
and thus by the triangle inequality
$$ \| e( T^h P ) - e(P) \|_{L^2(G)} < 2^{-k+3/2},$$
which we rearrange as
$$ \| e( \Delta_h P ) - 1 \|_{L^2(G)} < 2^{-(k-1)+1/2}.$$
By induction hypothesis, this implies that $\Delta_h P$ is constant for every $h$, thus $P$ is of degree $\leq 1$, and the claim then follows from the
$k=1$ case.
\end{proof}

One consequence of this lemma (and the second countability of $G$, which implies that $L^2(G)$ is separable) is that there are at most countably many polynomials $P$ of a given degree, up to constants and almost everywhere equivalence.

We now return to the proof of Theorem \ref{infty-group}.  We begin
with the easy case $k=1$.  Here we have $\|f\|_{L^\infty(G)}\leq 1$ and
$$ \left|\int_G f\ d\mu\right| \geq 1-\eps.$$
By rotating $f$ by a phase (which does not affect either the hypothesis or conclusion of the theorem) we may assume that
$$ \operatorname{Re} \int_G f\ d\mu \geq 1-\eps$$
and thus
$$ \int_G (1 - \operatorname{Re} f)\ d\mu = o(1).$$
By Markov's inequality, we thus have $1 - \operatorname{Re} f(x)= o(1)$ for $1-o(1)$ of the values of $x \in G$, which implies that $\|f-1\|_{L^1(G)} = o(1)$, and the claim follows.  For future reference we remark that the above argument did not use the group structure of $G$, and thus also establishes the $k=1$ case of Theorem \ref{infty-N} and Theorem \ref{infty-ergodic}.

Now assume inductively that $k \geq 2$, and that Theorem \ref{infty-group} has already been proven for $k-1$.  It will be convenient to relax the hypothesis $\|f\|_{U^k(G)} \geq 1-\eps$ to $\|f\|_{U^k(G)} \geq 1-o(1)$, as this will give us the freedom to perturb $f$ by $o(1)$ in the $L^{p_k}(G)$ (and hence $U^k(G)$) norm, subject of course to the constraint that the $L^\infty(G)$ norm of $f$ remains bounded above by $1$.

It will be convenient to adopt the following notation: a measurable subset of $G$ will be called \emph{very dense} if its measure is $1-o(1)$.

We may assume that $\eps$ is small, as the claim is vacuous for large $\eps$.  From \eqref{ukg2} we have
$$ \|f\|_{L^{p_k}(G)} = 1 - o(1)$$
and thus (since $\|f\|_{L^\infty(G)}\leq 1$), we have $|f|=1-o(1)$ for a very dense set of points $x$ in $G$.  Thus, by modifying $f$ by $o(1)$ in $L^{p_k}(G)$ norm, we may assume without loss of generality that $|f|=1$ everywhere.

Next, we use \eqref{fk} to observe that
$$ \int_G \|(T^h f) \overline{f} \|_{U^{k-1}(G)}^{2^{k-1}}\ d\mu(h) = 1-o(1).$$
Since
$$ \|(T^h f) \overline{f} \|_{U^{k-1}(G)} \leq \|(T^h f) \overline{f} \|_{L^\infty(G)} = 1,$$
we conclude from Markov's inequality that
$$ \|(T^h f) \overline{f} \|_{U^{k-1}(G)} = 1 - o(1)$$
for all $h$ in a very dense subset of $G$.  Applying the induction hypothesis, we see that there exists a very dense subset $A$ of $G$ such that for all $h \in A$, there exists a polynomial $P_h: G \to \R/\Z$ of degree $\leq k-2$ such that
$$ \|(T^h f) \overline{f} - e(P_h) \|_{L^1(G)} = o(1).$$

We now pass from the very dense set $A$ to all of $G$ as follows.  Let $h \in G$.  As $A$ and $h-A$ both have density greater than $1/2$ in $G$, they must intersect, thus we can write $h=a+b$ for some $a, b \in A$.  By hypothesis, one has
$$ \|(T^a f) \overline{f} - e(P_a) \|_{L^1(G)} = o(1)$$
and
$$ \|(T^b f) \overline{f} - e(P_b) \|_{L^1(G)} = o(1).$$
Since $|f|=1$, we conclude that
$$ \|T^a f - e(P_a) f \|_{L^1(G)} = o(1)$$
and
$$ \|T^{a+b} f - e(T^a P_b) T^a f \|_{L^1(G)} = o(1).$$
From the triangle inequality, we conclude that
$$ \|T^h f - e(T^a P_b + P_a) f \|_{L^1(G)} = o(1).$$
The expression $T^a P_b + P_a$ is a polynomial of degree $\leq k-2$.  We have thus shown that for \emph{every} $h \in G$, there exists a polynomial $Q_h: G \to \R/\Z$ of degree $\leq k-2$ such that
$$ \|T^h f - e(Q_h) f \|_{L^1(G)} = o(1).$$
Fix a $Q_h$ for each $h \in G$.  We may select $Q_h$ in a Borel measurable manner, because there are only countably many polynomials up to constants; see \cite[Appendix C]{bergelson-tao-ziegler}.

If $h, h' \in G$, then by repeating the previous argument we see that
$$ \| T^{h+h'} f - e(T^h Q_{h'} + Q_h) f \|_{L^1(G)} = o(1).$$
On the other hand,
$$ \| T^{h+h'} f - e(Q_{h+h'}) f \|_{L^1(G)} = o(1).$$
From the triangle inequality (and the fact that $|f|=1$) we thus have
$$ \| e( Q_{h+h'} - T^h Q_{h'} - Q_h ) - 1\|_{L^1(G)} = o(1)$$
which (by the boundedness of $e( Q_{h+h'} - T^h Q_{h'} - Q_h ) - 1$) implies that
\begin{equation}\label{qhh}
 \| e( Q_{h+h'} - T^h Q_{h'} - Q_h ) - 1\|_{L^2(G)} = o(1).
\end{equation}

For $\eps$ small enough, we thus conclude from Lemma \ref{pcon} that $Q_{h+h'} - T^h Q_{h'} - Q_h$ is constant for every $h,h' \in G$, thus
\begin{equation}\label{qhh2} Q_{h+h'} - T^h Q_{h'} - Q_h = c_{h,h'} \mod 1
\end{equation}
for some $c_{h,h'} \in \R$.  From \eqref{qhh} we can assume that $c_{h,h'} = o(1)$.  As $Q_h$ depends in a Borel measurable fashion on $h$, we see that $c_{h,h'}$ depends in a Borel measurable fashion on $h,h'$.

From \eqref{qhh2} we obtain the equation
$$ c_{h,h'} + c_{h+h',h''} = c_{h,h'+h''} + c_{h',h''} \mod 1$$
for all $h,h',h'' \in G$;  since the $c_{h,h'} = o(1)$, we can remove the $\mod 1$ projection here and conclude that
\begin{equation}\label{chh0}
c_{h,h'} + c_{h+h',h''} = c_{h,h'+h''} + c_{h',h''} 
\end{equation}
for all $h,h',h'' \in G$.

If we average \eqref{chh} in $h''$, we obtain a relation of the form
\begin{equation}\label{chh}
 c_{h,h'} = b(h) + b(h') - b(h+h')
\end{equation}
where $b(h) := \int_G c_{h,h''}\ d\mu(h'')$.  Since $c_{h,h''}=o(1)$, we have $b(h)=o(1)$ also.

\begin{remark}  In the language of cohomology for dynamical systems, the equation \eqref{chh0} asserts that the map $(h,h') \mapsto c_{h,h'}$ is a \emph{$2$-cocycle}, and the formula \eqref{chh} asserts that this map is in fact a \emph{$2$-coboundary}.  Thus we have shown that all sufficiently small $\R/\Z$-valued $2$-cocycles are $2$-coboundaries.  Once one moves away from the near-extremal regime, then genuinely non-trivial cocycles begin appearing, which ultimately forces one to replace polynomials by nilsequences and related objects; see \cite{host-kra}.
\end{remark}

Now set $\tilde Q_h := Q_h + b(h)$.  Then $\tilde Q_h$ is still a polynomial of degree $\leq k-2$, and we still have
\begin{equation}\label{thq}
 \| T^h f - e(\tilde Q_h) f \|_{L^1(G)} = o(1)
\end{equation}
for all $h \in G$.  From \eqref{qhh2} and \eqref{chh} we have the cocycle equation
$$ \tilde Q_{h+h'} = T^h \tilde Q_{h'} + \tilde Q_h$$
for all $h, h' \in G$.  Specialising this to $0$, we see that
$$ \tilde Q_{h'}(-h) = \phi(-h-h') - \phi(-h)$$
for all $h, h' \in G$, where $\phi(x) := \tilde Q_{-x}(0)$; thus
$$ \tilde Q_h := \Delta_h \phi.$$
As $\phi$ is Borel measurable and all derivatives $\Delta_h \phi$ of $\phi$ are polynomials of degree $\leq k-2$, $\phi$ is of degree $\leq k-1$.  We can then rewrite \eqref{thq} as
$$ \|T^h \tilde f - \tilde f \|_{L^1(G)} =o(1)$$
where $\tilde f:= f \cdot e(-\phi)$.  Averaging in $h$ using Minkowski's inequality we see that
$$ \| \tilde f - c \|_{L^1(G)} = o(1)$$
for some constant $c$.  As $|\tilde f|=1$, we may take $|c|=1$.  We then have
$$ \|f - c \cdot e(\phi)\|_{L^1(G)} = o(1)$$
and the claim follows.  This completes the proof of Theorem \ref{infty-group}.

\section{$L^\infty$ near-extremisers on intervals}\label{infty-n-sec}

We now prove Theorem \ref{infty-N}.  The arguments will closely follow the proof of Theorem \ref{infty-group} in the previous section, but localised to intervals such as $[N]$.  Again, we only treat the $\eps > 0$ case.

For $k=1$, the proof proceeds precisely as in the previous section, so we assume inductively that $k \geq 2$ and the claim has already been proven for $k-1$.

As in the previous section, we will just prove the near-extremiser claim, and relax the hypotheses to $\|f\|_{L^\infty([N])} \leq 1$ and $\|f\|_{U^k([N])} \geq 1-o(1)$.  We may assume that $\eps$ is positive but small, as the claim is trivial for large $\eps$.

It will be convenient to introduce the expectation notation $\E_{x \in A} f(x) := \frac{1}{|A|} \sum_{x \in A} f(x)$ for any finite non-empty set $A$.
We also say that a set $B$ is a \emph{very dense} subset of a finite non-empty set $A$ if one has $|B| \geq (1-o(1)) |A|$.

We embed $[N]$ in $\Z/\tilde N\Z$ for (say) $\tilde N := 2^k N + 1$, and extend $f$ by zero to a function from $\Z/\tilde N \Z \to \C$, which by abuse of notation we will still call $f$.  From \eqref{util} we have
$$ \|f\|_{U^k(\Z/\tilde N\Z)} \geq (1-o(1)) \|1_{[N]}\|_{U^k(\Z/\tilde N\Z)}$$
and hence by \eqref{fk}
$$ \E_{h \in \Z/\tilde N\Z} \| T^h f \overline{f} \|_{U^{k-1}(\Z/\tilde N\Z)} \geq (1-o(1))
\E_{h \in \Z/\tilde N\Z} \| 1_{h+[N]} 1_{[N]} \|_{U^{k-1}(\Z/\tilde N\Z)}.$$
The expectation on the right-hand side can be computed to be comparable to one (with bounds depending on $k$), and thus
$$  \E_{h \in \Z/\tilde N\Z}\| 1_{h+[N]} 1_{[N]} \|_{U^{k-1}(\Z/\tilde N\Z)} - \| T^h f \overline{f} \|_{U^{k-1}(\Z/\tilde N\Z)} = o(1).$$
As $\|f\|_{L^\infty(\Z/\tilde N\Z)} \leq 1$, the expression inside the expectation is non-negative.  Thus, by Markov's inequality, one has
$$ \| 1_{h+[N]} 1_{[N]} \|_{U^{k-1}(\Z/\tilde N\Z)} - \| T^h f \overline{f} \|_{U^{k-1}(\Z/\tilde N\Z)} = o(1)$$
for all $h$ in a very dense subset of $\Z/\tilde N\Z$.  In particular, for all $h$ in a very dense subset of $[N/2]$, one has
$$ \| T^{-h} f \overline{f} \|_{U^{k-1}(\Z/\tilde N\Z)} = \| 1_{[N-h]} \|_{U^{k-1}(\Z/\tilde N\Z)} - o(1).$$
The norm on the right-hand side can be calculated to be comparable to one (with bounds depending on $k$). If we thus let $f_h: [N-h] \to \C$ be the restriction of $T^{-h} f \overline{f}$ to $[N-h]$, we conclude from \eqref{util} that
$$ \|f_h\|_{U^{k-1}([N-h])} = 1-o(1).$$
Applying the induction hypothesis, we conclude that for all $h$ in a very dense subset $H$ of $[N/2]$, there exists a polynomial $P_h: \Z \to \R/\Z$ of degree $\leq k-2$ such that
\begin{equation}\label{fah}
 \|f_h - e(P_h) \|_{L^1([N-h])} = o(1).
\end{equation}

In particular, by the triangle inequality
$$ \| |f_h| - 1 \|_{L^1([N-h])} = o(1).$$
Using the crude bound $|f_h| \leq |f|, |T^{-h} f| \leq 1$, we thus have
$$ \E_{x \in [N-h]} 1 - |f(x)| = o(1)$$
and
$$ \E_{x \in [N-h]+h} 1 - |f(x)| = o(1)$$
for all $h \in H$, and so by the triangle inequality
$$ \E_{x \in [N]} 1-|f(x)| = o(1).$$
From this and Markov's inequality we conclude that $|f(x)|=1-o(1)$ for all $x$ in a very dense subset of $[N]$.  We may thus modify $f$ by $o(1)$ in $L^{2^k/(k+1)}([N])$ norm and assume without loss of generality that $|f(x)|=1$ for all $x \in [N]$.

We return now to \eqref{fah}, which we can now rewrite as
\begin{equation}\label{thaf0}
 \| T^{-h} f - e(P_h) f\|_{L^1([N-h])} = o(1)
\end{equation}
for all $h \in H$.

Let $h \in [-N/8,N/8]$.  Then, as $H$ contains a very dense subset of $[N/4]$, we can find $a, b \in H \cap [N/4]$ such that $b=a-h$.  From \eqref{thaf0} one has
$$ \| T^{-a} f - e(P_a) f \|_{L^1([3N/4])} = o(1)$$
and
$$ \| T^{-a+h} f - e(P_{a-h}) f \|_{L^1([3N/4])} = o(1)$$
and thus by the triangle inequality
$$ \| T^{-a+h} f - e( P_{a-h}-P_a) T^{-a} f \|_{L^1([3N/4])} = o(1)$$
and thus
$$ \| T^{h} f - e( T^a(P_{a-h}-P_a) ) f \|_{L^1([N/2])} = o(1).$$
Thus, for each $h \in [-N/8,N/8]$, we can choose a polynomial $Q_{h}$ of degree $\leq k-2$ such that
$$ \| T^{h} f - e( Q_{h} ) f \|_{L^1([N/2])} = o(1).$$
Now let $h, h' \in [-N/16,N/16]$.  Then we have
\begin{align*}
\| T^{h} f - e( Q_{h} ) f \|_{L^1([N/4])} &= o(1) \\
\| T^{h+h'} f - e( T^{h} Q_{h'} ) f \|_{L^1([N/4])} &= o(1) \\
\| T^{h+h'} f - e( Q_{h+h'} ) f \|_{L^1([N/4])} &= o(1)
\end{align*}
and thus by the triangle inequality and the fact that $|f|=1$ on $[N/4]$, we have
$$ \| e( Q_{h+h'} - T^h Q_{h'} - Q_h ) - 1\|_{L^2([N/4])} = o(1)$$
for all $h,h' \in [-N/16,N/16]$.

We now need a variant of Lemma \ref{pcon}:

\begin{lemma}[Separation lemma]\label{pcon2}  Let $k \geq 1$ and $N \geq 1$ be integers, and let $P: [N] \to \R/\Z$ be a polynomial $P(n) = \sum_{i=0}^k c_i n^i$ of degree $\leq k$ and coefficients $c_0,\ldots,c_k$ such that $\| e( P ) - 1\|_{L^2([N])} = o(1)$.  If $N$ is sufficiently large depending on $k$, then for each $0 \leq i \leq k$, one has $c_i = o(N^{-i}) \mod 1$.
\end{lemma}

\begin{proof}  For $k=1$ the claim follows by direct calculation of geometric series (or by Fourier analysis), so suppose that $k \geq 2$ and that the claim has already been proven for $k-1$.  Arguing exactly as in the proof of Lemma \ref{pcon}, we see that
$$ \| e( \Delta_h P ) - 1 \|_{L^2([N/2])} = o(1)$$
for all $h \in [-N/2,N/2]$.  By induction hypothesis, this implies that for each $1 \leq i \leq k-1$, the $i^{th}$ Fourier coefficient of $\Delta_h P$ is $o(N^{-i}) \mod 1$.  When $i=k-1$, this coefficient is $k h c_k$, thus $khc_k = o(N^{-k+1}) \mod 1$ for all $h \in [-N/2,N/2]$.  We conclude that $c_k = q_k + o(N^{-k}) \mod 1$, where $q_k$ is a multiple of $1/k$.  If we then turn to $i=k-2$, and restrict $h$ to be a multiple of $k$ to eliminate the contribution of $q_k$, we see that the $(k-2)^{th}$ coefficient of $\Delta_h P$ is $(k-1) h c_{k-1} + o(N^{-k+2}) \mod 1$.  Repeating the previous argument, we see that $c_{k-1} = q_{k-1} + o(N^{-k}) \mod 1$ where $q_{k-1}$ is a multiple of $1/k(k-1)$.  Iterating this, we eventually obtain $c_i = q_i + o(N^{-i}) \mod 1$ for all $1 \leq i \leq k$ and some $q_i$ that is a multiple of $1/k!$.  The $o(N^{-i})$ component of $c_i$ only influences $e(P)$ by $o(1)$ on $[N]$, so without loss of generality we may assume that $c_i = q_i$, t
 hus $P$ is now periodic with period $k!$, and can be viewed as a function on $\Z/k!\Z$.  If $N$ is large enough (e.g. $N \geq k!$ will suffice), this implies that
$$ \|e(P)-1\|_{L^2(\Z/k!\Z)} = o(1).$$
The claim now follows from Lemma \ref{pcon}.
\end{proof}

Applying this lemma, we thus see that for all $h,h' \in [-N/16,N/16]$, the degree $\leq k-2$ polynomial
\begin{equation}\label{chha}
 c_{h,h'} := Q_{h+h'} - T^h Q_{h'} - Q_h
\end{equation}
is nearly constant in the sense that its $i^{th}$ coefficient is $o(N^{-i})$ for each $0 \leq i \leq k-2$.

From \eqref{chha} we see that the $c_{h,h'}$ obey the $2$-cocycle equation
$$ c_{h,h'} + c_{h+h',h''} = c_{h,h'+h''} + T^h c_{h',h''};$$
this is similar to the situation in the previous situation, but the $c_{h,h'}$ are now polynomials instead of constants (which, among other things, introduces the shift $T^h$ in the above equation).  However, one can remove the polynomial coefficients as follows:

\begin{lemma}  For any $-1 \leq j \leq k-2$, there exists $0 < \delta_j \leq 1/16$ depending only on $j$ and $k$, and a polynomial $\tilde Q_h$ of degree $\leq k-2$ for each $h \in [-\delta_j N, \delta_j N]$, such that
\begin{equation}\label{thqh}
\| T^{h} f - e( \tilde Q_{h} ) f \|_{L^1([N/2])} = o(1)
\end{equation}
for all $h \in [-\delta_j N, \delta_j N]$, and such that
$$ \tilde c_{h,h'} := \tilde Q_{h+h'} - T^h \tilde Q_{h'} - \tilde Q_h$$
is a degree $\leq j$ polynomial with an $i^{th}$ coefficient of $o(N^{-i})$ for each $1 \leq i \leq j$ and all $h \in [-\delta_j N/2, \delta_j N/2]$.
\end{lemma}

\begin{proof} We perform downward induction on $j$.  The case $j=k-2$ follows from the previous discussion.  Now suppose that $0 \leq j < k-2$ and that the claim has already been proven for $j+1$.  Let $\tilde Q_h, \tilde c_{h,h'}, \delta_{j+1}$ be the objects generated by the induction hypothesis.  If we let $\tilde c_{h,h'}^{(j+1)}$ be the $(j+1)^{th}$ coefficient of $\tilde c_{h,h'}$, identified with an element of the fundamental domain $(-1/2,1/2]$, then we have $\tilde c_{h,h'}^{(j+1)} = o(N^{-j-1})$, and we have the $2$-cocycle equation
$$ \tilde c_{h,h'}^{(j+1)} + \tilde c_{h+h',h''}^{(j+1)} = \tilde c_{h,h'+h''}^{(j+1)} + \tilde c_{h',h''}^{(j+1)}$$
for all $h,h',h'' \in [-\delta_{j+1} N/2, \delta_{j+1} N/2]$.  We average this over all $h'' \in [-\delta_{j+1} N/2, \delta_{j+1} N/2]$.  If $h' = 1$, we have
$$ \E_{h'' \in [-\delta_{j+1} N/2, \delta_{j+1} N/2]} \tilde c_{h,1+h''}^{(j+1)} = \E_{h'' \in [-\delta_{j+1} N/2, \delta_{j+1} N/2]} \tilde c_{h,h''}^{(j+1)} + o(N^{-j-2}),$$
and so we have
$$ \tilde c_{h,1}^{(j+1)} + a(h+1) = a(h) + a(1) + o(N^{-j-2})$$
for $h \in [-\delta_{j+1} N/2,\delta_{j+1} N/2]$, where
$$
a(h) := \E_{h'' \in [-\delta_{j+1} N/2, \delta_{j+1} N/2]} \tilde c_{h,h''}^{(j+1)}.$$
From the bounds on $\tilde c_{h,h''}^{(j+1)}$, we have $a(h) = o(N^{-j-1})$.  If we then set $Q'_h(n) := \tilde Q_h(n) + \frac{1}{j+1} a(h) n^{j+1}$ to $\tilde Q_h(n)$, then $Q'_h$ also obeys the required property \eqref{thqh} (with a slightly different value of $o(1)$), and the $(j+1)^{th}$ coefficient $(c'_{h,h'})^{(j+1)}$ of the degree $\leq j+1$ polynomial
$$ c'_{h,h'} := Q'_{h+h'} - T^h Q'_{h'} - Q'_h$$
is given by
$$ (c'_{h,h'})^{(j+1)} = (\tilde c_{h,h'})^{(j+1)} + a(h+h') - a(h) - a(h')$$
and thus
$$ (c')_{h,1}^{(j+1)} = o(N^{-j-2})$$
for $h \in [-\delta_{j+1} N/2,\delta_{j+1} N/2]$.  Of course, $(c')_{h,h'}^{(j+1)}$ also obeys the cocycle equation
$$ (c'_{h,h'})^{(j+1)} + (c'_{h+h',h''})^{(j+1)} = (c'_{h,h'+h''})^{(j+1)} + (c'_{h',h''})^{(j+1)}$$
and vanishes when $h$ or $h'$ equals zero.

Now define $b(h)$ for $h \in [-\delta_{j+1} N/4,\delta_{j+1} N/4]$ by requiring that
$$ b(0) = 0$$
and
$$ b(h+1) = b(h) + (c')_{h,1}^{(j+1)}$$
for all $h \in [-\delta_{j+1} N/4,\delta_{j+1} N/4]$, thus $b(h) = o(N^{-j-1})$ for all $h \in [-\delta_{j+1} N/4,\delta_{j+1} N/4]$.

From the cocycle equation and induction on $h'$ one verifies that
$$ (c'_{h,h'})^{(j+1)} = b(h+h') - b(h) - b(h')$$
for all $h, h' \in [-\delta_{j+1} N/8,\delta_{j+1} N/8]$.    If we thus set
$$ Q''_h(n) := Q'_h(n) - \frac{1}{j+1} b(h) n^{j+1}$$
for all $h \in [-\delta_{j+1} N/8, \delta_{j+1} N/8]$, then $Q''_h$ also obeys the required property \eqref{thqh} (with a slightly different value of $o(1)$), and the $(j+1)^{th}$ coefficient of the degree $\leq j+1$ polynomial $Q''_{h+h'} - T^h Q''_{h'} - Q''_h$ vanishes for $h,h' \in [-\delta_{j+1} N/8,\delta_{j+1} N/8]$, so that this polynomial in fact has degree $\leq j$.  The claim follows (setting $\delta_j := \delta_{j+1}/8$).
\end{proof}

We apply this lemma with $j=-1$ to obtain polynomials $\tilde Q_h$ of degree $\leq k-2$ for each $h \in [-\delta_{-1} N, \delta_{-1} N]$ obeying \eqref{thqh}
and such that
$$ \tilde Q_{h+h'} - T^h \tilde Q_{h'} - \tilde Q_h = 0$$
for all $h, h' \in [-\delta_{-1} N/2, \delta_{-1} N/2]$.  In particular, setting $F(h) :=\tilde Q_h(0)$ for $h \in [-\delta_{-1} N/2, \delta_{-1} N/2]$, we have
$$ \tilde Q_h(x) = \Delta_h F(x)$$
for all $h, x \in [-\delta_{-1} N/4, \delta_{-1} N/4]$.  As each $\tilde Q_h$ is a polynomial of degree $\leq k-2$, this implies that
$$ \Delta_{h_1} \ldots \Delta_{h_{k}} F(x) = 0$$
whenever $h, x \in [-\delta_{-1} N/4k, \delta_{-1} N/4k]$.  Setting $h_1=\ldots=h_k=1$ and inducting on $k$, we conclude that there exists a polynomial $\tilde F: \Z \to \R/\Z$ of degree $\leq k-1$, thus
$$ \tilde F(x) = \sum_{i=0}^{k-1} f_i x^i$$
for some $f_0,\ldots,f_{k-1} \in \R/\Z$, such that $F(x)=\tilde F(x)$ for all $x \in [-\delta_{-1} N/4k, \delta_{-1} N/4k]$.  We conclude that $\tilde Q_h$ and $\Delta_h \tilde F$ agree on $[-\delta_{-1} N/8k, \delta_{-1} N/8k]$ for all $h \in [-\delta_{-1} N/8k, \delta_{-1} N/8k]$. By polynomial interpolation, this implies that $\tilde Q_h = \Delta_h \tilde F$ for all $h \in [-\delta_{-1} N/8k, \delta_{-1} N/8k]$.  From \eqref{thqh} we thus have
$$
\| T^{h} f - e( \Delta_h \tilde F ) f \|_{L^1([N/2])} = o(1)
$$
for all $h \in [-\delta_{-1} N/8k, \delta_{-1} N/8k]$; if we then set $\tilde f := f \cdot e(-\tilde F)$, then we have
$$
\| T^{h} \tilde f - \tilde f \|_{L^1([N/2])} = o(1)
$$
for all $h \in [-\delta_{-1} N/8k, \delta_{-1} N/8k]$.

Note that if we multiply $f$ by a polynomial phase such as $e(-\tilde F)$, then both the hypothesis and conclusion of Theorem \ref{infty-N} remain unchanged (and $f$ remains of magnitude $1$).  Thus we may assume without loss of generality that $\tilde f=f$, thus
$$
\| T^{h} f - f \|_{L^1([N/2])} = o(1)
$$
and so (recalling the definition of $f_h = T^h f \overline{f}$),
$$
\| f_h - 1 \|_{L^1([N/2])} = o(1).
$$
Comparing this with \eqref{fah}, we see that for all $h$ in a very dense subset of $[\delta_{-1} N/8k]$, one has
$$ \| e(P_h) - 1 \|_{L^1([N/2])} = o(1);$$
applying Lemma \ref{pcon2}, we conclude that
$$ \| e(P_h) - 1 \|_{L^1([N-h])} = o(1)$$
and hence that
$$
\| f_h - 1 \|_{L^1([N-h])} = o(1)
$$
or equivalently
$$
\| T^h f - f \|_{L^1([N-h])} = o(1)
$$
for a all $h$ in a very dense subset $H$ of $[-\delta_{-1} N/8k]$.

Observe that if $h, h' \in H$, then by the triangle inequality we have
$$
\| T^{h+h'} f - f \|_{L^1([N-h-h'])} = o(1).
$$
Iterating this a bounded number of times, we conclude that
$$ \|T^h f - f \|_{L^1([N-h])} = o(1)$$
for all $h$ in a very dense subset of $[N]$.  Thus
$$ \E_{x \in \Z/\tilde N\Z} T^h f \overline{f}(x) - T^h 1_{[N]} 1_{[N]}(x) = o(1)$$
for all $h$ in a very dense subset of $[N]$.  By reflection symmetry, we may extend this to all $h$ in a very dense subset of $[-N,N] = \{-N,\ldots,N\}$, and then (by the support properties of $f$) to a very dense subset of $\Z/\tilde N\Z$.  Averaging over $h$, we then see that
$$ \E_{x,h \in \Z/\tilde N\Z} T^h f \overline{f}(x) - T^h 1_{[N]} 1_{[N]}(x) = o(1)$$
which simplifies to
$$ |\E_{x \in \Z/\tilde N\Z} f(x)|^2 - |\E_{x \in \Z/\tilde N\Z} 1_{[N]}(x)|^2 = o(1)$$
and hence
$$ \|f\|_{U^1([N])} = 1 - o(1).$$
The claim now follows from the $k=1$ case of the theorem.

\section{$L^\infty$ near-extremisers on systems}\label{infty-system-sec}

We now prove Theorem \ref{infty-ergodic}.  Unsurprisingly, the arguments will closely follow the proofs of Theorem \ref{infty-group} and Theorem \ref{infty-N}. 

For $k=1$, the proof proceeds precisely as in the previous sections, so we assume inductively that $k \geq 2$ and the claim has already been proven for $k-1$.

As in previous sections, we will just prove the near-extremiser claim, and relax the hypotheses to $\|f\|_{L^\infty(X)} \leq 1$ and $\|f\|_{U^k(X)} \geq 1-o(1)$.  We may assume that $\eps$ is positive but small, as the claim is trivial for large $\eps$.

We call a subset $A$ of the natural numbers \emph{very dense} if the lower density $\liminf_{H \to\infty} \frac{1}{H} |A \cap [H]|$ is $1-o(1)$.

From \eqref{ukp} we have
$$ \|f\|_{L^{p_k}(X)} = 1 - o(1),$$
so by arguing exactly as in Section \ref{infty-group-sec}, we may assume without loss of generality that $|f|=1$ everywhere.

From \eqref{usx} one has
$$ \lim_{H \to \infty} \E_{h \in [H]} \|(T^h f) \overline{f} \|_{U^{k-1}(X)}^{2^{k-1}} = 1-o(1).$$
Since
$$ \|(T^h f) \overline{f} \|_{U^{k-1}(X)} \leq \|(T^h f) \overline{f} \|_{L^\infty(X)} = 1,$$
we conclude from Markov's inequality that
$$ \|(T^h f) \overline{f} \|_{U^{k-1}(X)} = 1 - o(1)$$
for all $h$ in a very dense subset of $\N$.  Applying the induction hypothesis, we see that there exists a very dense subset $A$ of $\N$ such that for all $h \in A$, there exists a polynomial $P_h: X \to \R/\Z$ of degree $\leq k-2$ such that
$$ \|(T^h f) \overline{f} - e(P_h) \|_{L^1(X)} = o(1).$$

As in Section \ref{infty-group-sec}, the next step is to pass from $A$ to all of $\Z$.  Let $h \in \Z$.  As $A$ and $h+A$ both have lower density $1-o(1)$, we thus must have a representation $h=a-b$ for some $a, b \in A$.  By hypothesis, one has
$$ \|(T^a f) \overline{f} - e(P_a) \|_{L^1(X)}, \|(T^b f) \overline{f} - e(P_b) \|_{L^1(X)} = o(1).$$
Since $|f|=1$, we conclude that
$$ \| T^{a-b} f - e(T^{-b} P_a) T^{-b} f \|_{L^1(X)}, \| T^{-b} f - e(-T^{-b} P_b) f \|_{L^1(X)} = o(1)$$
and thus by the triangle inequality
$$ \| T^{a-b} f - e(T^{-b} P_a - T^{-b} P_b) f\|_{L^1(X)} = o(1).$$
The expression $T^{-b} P_a - T^{-b} P_b$ is a polynomial of degree $\leq k-2$.  We have thus shown that for \emph{every} $h \in \Z$, there exists a polynomial $Q_h: X \to \R/\Z$ of degree $\leq k-2$ such that
$$ \|T^h f - e(Q_h) f \|_{L^1(X)} = o(1).$$
Fix a $Q_h$ for each $h \in \Z$.

If $h, h' \in \Z$, then by arguing as in Section \ref{infty-group-sec} one has
\begin{equation}\label{qhhx}
 \| e( Q_{h+h'} - T^h Q_{h'} - Q_h ) - 1\|_{L^2(X)} = o(1).
\end{equation}

By the analogue of Lemma \ref{pcon} for systems (see \cite[Lemma C.1]{bergelson-tao-ziegler}), we conclude that $Q_{h+h'} - T^h Q_{h'} - Q_h$ is constant for every $h,h' \in \Z$, thus
\begin{equation}\label{qhh2x} Q_{h+h'} - T^h Q_{h'} - Q_h = c_{h,h'} \mod 1
\end{equation}
for some $c_{h,h'} \in \R$.  From \eqref{qhhx} we can assume that $c_{h,h'} = o(1)$.

From \eqref{qhh2x} we obtain the $2$-cocycle equation
$$ c_{h,h'} + c_{h+h',h''} = c_{h,h'+h''} + c_{h',h''} \mod 1$$
for all $h,h',h'' \in \Z$;  since the $c_{h,h'} = o(1)$, we can remove the $\mod 1$ projection here and conclude that
$$ c_{h,h'} + c_{h+h',h''} = c_{h,h'+h''} + c_{h',h''} $$
for all $h,h',h'' \in \Z$.

We perform the averaging trick.  As $\Z$ is amenable, there is a translation-invariant mean $\lambda: \ell^\infty(\Z) \to \R$ on the bounded real sequences on $\Z$.  Applying this mean to average in $h''$, we obtain the relation
\begin{equation}\label{chhx}
 c_{h,h'} = b(h) + b(h') - b(h+h')
\end{equation}
for each $h,h' \in \Z$, where $b(h) = o(1)$ depends only on $h$.  Now set $\tilde Q_h := Q_h + b(h)$, then $\tilde Q_h$ is still a polynomial of degree $\leq k-2$, and we still have
\begin{equation}\label{thqx}
 \| T^h f - e(\tilde Q_h) f \|_{L^1(X)} = o(1)
\end{equation}
for all $h \in \Z$.  From \eqref{qhh2x} and \eqref{chhx} we have the cocycle equation
\begin{equation}\label{coc}
 \tilde Q_{h+h'} = T^h \tilde Q_{h'} + \tilde Q_h
\end{equation}
for all $h, h' \in \Z$.

We rewrite \eqref{thqx} as
$$ \| f - e(-\tilde Q_h) T^h f \|_{L^1(X)} = o(1).$$
By the triangle inequality, we thus have
\begin{equation}\label{thof}
 \| f - \E_{h \in [H]} e(-\tilde Q_h) T^h f \|_{L^1(X)} = o(1).
\end{equation}
We now claim that the expression $\E_{h \in [H]} e(-\tilde Q_h) T^h f$ converges in $L^1(X)$ norm to a limit $F$.  This can be seen from the mean ergodic theorem. Indeed, if one considers the cocycle extension $\tilde X := X \times \R/\Z$ with shift $\tilde T: \tilde X \to \tilde X$ defined by
$$ \tilde T( x, \theta ) := (Tx, \theta + \tilde Q_1(x) )$$
then from \eqref{coc} and induction one has
$$ \tilde T^h( x, \theta ) := (T^hx, \theta + \tilde Q_h(x) )$$
for every $h \in \Z$.  If one then lets $\tilde f \in L^\infty(\tilde X)$ be the function
$$ \tilde f( x, \theta ) := f(x) e(-\theta)$$
then we have
\begin{equation}\label{thaf}
 \E_{h \in [H]} \tilde T^h \tilde f(x,\theta) = \E_{h \in [H]} e(-\tilde Q_h(x)) T^h f(x) e(-\theta).
\end{equation}
By the mean ergodic theorem, the left-hand side converges in $L^1(\tilde X)$ to a $\tilde T$-invariant function $\tilde F$, which must then take the form $F(x) e(-\theta)$ by an inspection of the right-hand side of \eqref{thaf}, and the claim follows.

From \eqref{thof} one has
\begin{equation}\label{fF}
\| f - F \|_{L^1(X)} = o(1).
\end{equation}
From the $\tilde T$-invariance of $\tilde F$, one has
\begin{equation}\label{thf}
 T^h F = e( \tilde Q_h ) F
\end{equation}
for all $h \in \Z$.  In particular, $|F|$ is $T$-invariant and thus constant almost everywhere by ergodicity.  One can then write $F = c \cdot e(P)$ almost everywhere for some constant $c$ and some measurable $P: X \to \R/\Z$.  From \eqref{fF} one has $|c|=1-o(1)$, and in particular $c$ is non-zero.  From \eqref{thf} we then have $T^h P - P = \tilde Q_h$ for all $h \in \Z$, and in particular $P$ is polynomial of degree $\leq k-1$.  We now have $\|f-e(P)\|_{L^1(X)} = o(1)$, and the claim follows.

\section{Near-extremisers of the H\"older and Young inequalities}

Let $(X,{\mathcal X},\mu)$ be a measure space.  The H\"older inequality asserts that
$$ \|f_1 \ldots f_m\|_{L^{p}(X)} \leq \|f_1\|_{L^{p_1}(X)} \ldots \|f_m\|_{L^{p_m}(X)}$$
whenever $0 < p_1,\ldots,p_m,p \leq \infty$ are such that $\frac{1}{p_1}+\ldots+\frac{1}{p_m} = \frac{1}{p}$ and $f_i \in L^{p_i}(X)$ for $i=1,\ldots,m$.  If the $p_1,\ldots,p_m,p$ are finite and the $f_i$ are not almost everywhere zero, it is well known that equality occurs if and only if the $|f_i|^{p_i}$ are all constant multiples of each other up to almost everywhere equivalence.

We now establish a stable version of the above assertion:

\begin{lemma}[Near-extremisers of H\"older]\label{nearext}  Let $m \geq 2$, let $(X,{\mathcal X},\mu)$ be a measure space, and let $0 < p_1,\ldots,p_m,p < \infty$ be such that $\frac{1}{p_1}+\ldots+\frac{1}{p_m} = \frac{1}{p}$, and for each $i=1,\ldots,m$, let $f_i \in L^{p_i}(X)$ be such that $\|f_i\|_{L^{p_i}(X)} > 0$, and such that
$$ \|f_1 \ldots f_m\|_{L^{p}(X)} \geq (1-\eps) \|f_1\|_{L^{p_1}(X)} \ldots \|f_m\|_{L^{p_m}(X)}$$
for some $\eps > 0$.  Then there exists a measurable subset $E$ of $X$ with
$$ \int_E |f_i|^{p_i}\ d\mu = o_{m,p_1,\ldots,p_m}( \|f_i\|_{L^{p_i}(X)}^{p_i} )$$
for all $i=1,\ldots,m$, and positive real numbers $c_1,\ldots,c_m > 0$ such that
$$ c_i |f_i(x)|^{p_i} = (1+o_{m,p_1,\ldots,p_m}(1)) |f_1 \ldots f_m(x)|^{p}.$$
for all $x \in X \backslash E$ and $i = 1,\ldots,m$.  In particular, one has
$$ c_i |f_i(x)|^{p_i} = (1+o_{m,p_1,\ldots,p_m}(1)) c_j |f_j(x)|^{p_j}$$
for all $x \in X \backslash E$ and $i,j = 1,\ldots,m$.
\end{lemma}

To prove this lemma, we first need a simple measure-theoretic lemma:

\begin{lemma}\label{mes}  Let $(X,{\mathcal X},\mu)$ be a measure space, and let $f, g: X \to \R^+$ be non-negative absolutely integrable functions.  Suppose that
$$ \int_X g\ d\mu \leq \eps \int_X f\ d\mu$$
for some $\eps > 0$.  Then there exists a set $E$ with
$$ \int_E f\ d\mu < \sqrt{\eps} \int_X f\ d\mu$$
such that
$$ g(x) \leq \sqrt{\eps} f(x)$$
for all $x \in X \backslash E$.
\end{lemma}

\begin{proof} Set $E := \{ x \in X: g(x) > \sqrt{\eps} f(x) \}$.  Then
$$ \int_E f\ d\mu < \frac{1}{\sqrt{\eps}} \int_X g\ d\mu \leq \sqrt{\eps} \int_X f\ d\mu$$
and the claim follows.
\end{proof}

\begin{proof}[Proof of Lemma \ref{nearext}] We abbreviate $o_{m,p_1,\ldots,p_m}(1)$ as $o(1)$.  By induction on $m$, it suffices to verify the case $m=2$.  By replacing $f_i$ with $|f_i|^p$ for $i=1,2$, we may assume that $p=1$ (thus $1/p_1+1/p_2 = 1$) and that $f_1, f_2$ are non-negative.  By homogeneity we may rescale so that $\|f_1\|_{L^{p_1}(X)} = \|f_2\|_{L^{p_2}(X)} = 1$.  We then have
$$ \int_X f_1 f_2(x)\ d\mu(x) \geq 1-\eps.$$
On the other hand, we have the elementary inequality
$$ \int_X f_1 f_2(x)\ d\mu(x) \leq \int_X \left(\frac{1}{p_1} f_1^{p_1}(x) + \frac{1}{p_2} f_1^{p_2}(x)\right)\ d\mu(x) = 1.$$
We thus have
$$ \int_X \left(\frac{1}{p_1} f_1^{p_1}(x) + \frac{1}{p_2} f_1^{p_2}(x) - f_1 f_2(x)\right)\ d\mu(x) = o(1).$$
As the integrand is non-negative, we can thus apply Lemma \ref{mes} to obtain a measurable set $E$ such that
$$ \int_E \left(\frac{1}{p_1} f_1^{p_1}(x) + \frac{1}{p_2} f_1^{p_2}(x)\right)\ d\mu(x) = o(1)$$
and
\begin{equation}\label{expab}
 \frac{1}{p_1} f_1^{p_1}(x) + \frac{1}{p_2} f_1^{p_2}(x) - f_1 f_2(x) = o\left(\frac{1}{p_1} f_1^{p_1}(x) + \frac{1}{p_2} f_1^{p_2}(x)\right)
 \end{equation}
for all $x \notin E$.  Writing $f_1^{p_1}(x) = \exp(a)$ and $f_2^{p_2}(x) = \exp(b)$, one can rewrite \eqref{expab} as
$$ \exp\left( \frac{1}{p_1} a + \frac{1}{p_2} b \right) = (1-o(1)) \left(\frac{1}{p_1} \exp(a) + \frac{1}{p_2} \exp(b)\right).$$
Using the convexity of the exponential function (and normalising to, say, $\max(a,b)=0$ if desired) we conclude that $a=b+o(1)$, or in other words that $f_1^{p_1}(x) = (1+o(1)) f_2^{p_2}(x)$, and the claim follows.
\end{proof}

We now use the above lemma to analyse near-extremisers to Young's inequality.  We begin by recalling this inequality, together with its proof via H\"older's inequality:

\begin{proposition}[Young's inequality]\label{young}  Let $G = (G,+,{\mathcal B},\mu)$ be a locally compact abelian group, and let $0 < r < p,q < s < \infty$ be such that $\frac{1}{r}+\frac{1}{s}=\frac{1}{p}+\frac{1}{q}$.  Then for any $f \in L^p(G)$ and $g \in L^q(G)$, one has
\begin{equation}\label{young-ineq}
\left (\int_G \|(T^h f) g\|_{L^r(G)}^s\ d\mu(h) \right)^{1/s}\leq \|f\|_{L^p(G)} \|g\|_{L^q(G)}.
\end{equation}
\end{proposition}

In practice, we will only apply this inequality with the exponents
\begin{equation}\label{pqk}
p = q = p_k = \frac{2^{k}}{k+1}; \quad r = p_{k-1} = \frac{2^{k-1}}{k}; \quad s = 2^{k-1}
\end{equation}
for some $k \geq 2$; one easily verifies that the hypotheses of this proposition are satisfied.

\begin{proof}
Let $h \in G$. From the factorisation
$$
|T^h f| |g| = |T^h f|^{\frac{s-p}{s}} |g|^{\frac{s-q}{s}} (|T^h f(x)|^p |g(x)|^q)^{1/s},$$
H\"older's inequality and the identity
$$ \frac{1}{r} = \frac{s-p}{sp} + \frac{s-q}{sq} + \frac{1}{s}$$
one obtains
\begin{equation}\label{hold}
\| |T^h f| |g| \|_{L^r(G)} \leq \| |T^h f|^{\frac{s-p}{s}} \|_{L^{sp/(s-p)}(G)} \| |g|^{\frac{s-q}{s}} \|_{L^{sq/(s-q)}(G)} \| (|T^h f(x)|^p |g(x)|^q)^{1/s} \|_{L^s(G)}
\end{equation}
which we rearrange as
$$ \|(T^h f) g\|_{L^r(G)}^s \leq \|f\|_{L^p(G)}^{s-p} \|g\|_{L^q(G)}^{s-q} \int_G |T^h f(x)|^p |g(x)|^q\ d\mu(x)$$
for any $h \in G$.  On the other hand, from Fubini's theorem we have
\begin{equation}\label{sauce}
 \int_G \int_G |T^h f(x)|^p |g(x)|^q\ d\mu(x) d\mu(h) = \|f\|_{L^p(G)}^p \|g\|_{L^q(G)}^q,
\end{equation}
and \eqref{young-ineq} follows.
\end{proof}

If $H$ is a compact open subgroup of $G$, and $x_0+H, x'_0+H$ are two cosets of $H$, then a brief calculation shows that equality will hold in \eqref{young-ineq} if $f$ is a scalar multiple of $1_{x_0+H}$, and $g$ is a scalar multiple of $1_{x'_0+H}$.  It was observed in \cite{fournier} that this is essentially the only such example.  More precisely, we have:

\begin{proposition}[Inverse Young inequality]\label{iyi}\cite{fournier}  Let the notation and hypotheses be as in Proposition \ref{young}.  Assume the normalisation
\begin{equation}\label{normp}
 \|f\|_{L^p(G)} = \|g\|_{L^q(G)} = 1
\end{equation}
and assume that
\begin{equation}\label{tfg}
 \left( \int_G \|(T^h f) g\|_{L^r(G)}^s\ d\mu(h)\right)^{1/s} \geq 1-\eps
\end{equation}
for some $\eps > 0$.  Then there exists a compact open subgroup $H$ of $G$ and cosets $x_0+H, x'_0+H$ such that
$$ \left\||f|- \mu(H)^{-1/p} 1_{x_0+H}\right\|_{L^p(G)} = o(1); \quad \left\||g|- \mu(H)^{-1/q} 1_{x'_0+H} \right\|_{L^q(G)} = o_{p,q,r,s}(1).$$
\end{proposition}

This result does not explicitly appear in \cite{fournier}, but follows from the methods in that paper.
For the convenience of the reader we now give the proof of this proposition here.  We first need an inverse sumset estimate:

\begin{lemma}[Inverse sumset estimate]\label{ise} Let $G$ be a compact abelian group with Haar measure $\mu$, and let $K \subset G$ be a compact set of positive measure such that
$$ \mu(K-K) < \frac{3}{2} \mu(K)$$
where $K-K := \{ k-k': k,k' \in K \}$ is the difference set of $K$ with itself.  Then $K-K$ is a compact open subgroup of $G$.
\end{lemma}

\begin{proof}
Let $a$ and $b$ both be elements of $K-K$.  Writing $a$ as $a=k-k'$, we see that $K-k$, $K-k'$ both lie in $K-K$, which has measure less than $\frac{3}{2} \mu(K)$, and so $K-k$ and $K-k'$ must intersect in a set of measure greater than $\frac{1}{2} \mu(K)$; we conclude that $\mu(K \cap (K+a))
> \frac{1}{2} \mu(K)$.  Similarly $\mu(K \cap (K+b)) > \frac{1}{2} \mu(K)$, and hence $\mu((K+a) \cap (K+a+b)) > \frac{1}{2} \mu(K)$.  By the triangle inequality, this forces $\mu(K \cap (K+a+b)) > 0$, and in particular $K \cap (K+a+b)$ is non-empty.  In other words, $a+b \in K-K$.  We thus see that $K-K$ is closed under addition; as it is also symmetric, it is a group.  As $K-K$ is compact and has positive measure, it must also be open, and the claim follows.
\end{proof}

Now we can prove Proposition \ref{iyi}.

\begin{proof}[Proof of Proposition \ref{iyi}] We abbreviate $o_{p,q,r,s}(1)$ as $o(1)$.
We may assume that $\eps$ is small, as the claim is trivial for large $\eps$.  Observe that we may relax the right-hand side of \eqref{tfg} from $1-\eps$ to $1-o(1)$.  We may similarly relax the normalisations \eqref{normp} to
$$
 \|f\|_{L^p(G)}, \|g\|_{L^q(G)} = 1 + o(1).
$$
The point of doing this relaxation is that the hypotheses (and conclusion) of the proposition are now stable under perturbations of $f$ by $o(1)$ in the $L^p$ norm and perturbations of $g$ by $o(1)$ in the $L^q$ norm (the stability of \eqref{tfg} following from Proposition \ref{young}).  We will exploit this freedom to perturb $f, g$ by $o(1)$ in these norms later in this argument.

By replacing $f, g$ with $|f|, |g|$ respectively, we may assume without loss of generality that $f, g$ are non-negative.
Define the non-negative functions $A, B: G \to \R^+$ by the formulae
$$ A(h) := \|(T^h f) g\|_{L^r(G)}^s$$
and
$$ B(h) := \int_G (T^h f(x))^p g(x)^q\ d\mu(x).$$
Clearly $A, B$ are measurable (in fact, they are continuous).  From the proof of Proposition \ref{young} and the hypotheses, we have
$$ 0 \leq A(h) \leq (1+ o(1))B(h)$$
for all $h$, and
\begin{equation}\label{goo}
 1-o(1) \leq \int_G A(h)\ d\mu(h) \leq (1+o(1)) \int_G B(h)\ d\mu(h) \leq 1+o(1).
\end{equation}
Thus there exists $h$ such that
$$ 0 < (1-o(1)) B(h) < A(h) \leq (1+ o(1))B(h).$$
By translating $f$ by $h$ if necessary (which does not materially affect either the hypotheses or conclusions of Proposition \ref{iyi}), we may assume without loss of generality that $h=0$.  Thus we have
\begin{equation}\label{fgo}
 \| f g \|_{L^r(G)} > (1-o(1)) \| f^{\frac{s-p}{s}} \|_{L^{sp/(s-p)}(G)} \| g^{\frac{s-q}{s}} \|_{L^{sq/(s-q)}(G)} \| (f^p g^q)^{1/s} \|_{L^s(G)}.
\end{equation}
Applying Lemma \ref{nearext}, we can thus find a measurable set $E$ with
$$ \int_E f^p\ d\mu, \int_E g^q\ d\mu = o(1)$$
and constants $c_1, c_2, c_3 > 0$ such that
\begin{align*}
f^p(x) &= (1+o(1)) c_1 f^r(x) g^r(x) \\
g^q(x) &= (1+o(1)) c_2 f^r(x) g^r(x) \\
f^p(x) g^q(x) &= (1+o(1)) c_3 f^r(x) g^r(x)
\end{align*}
for all $x \in G \backslash E$.  By modifying $f, g$ by $o(1)$ in the $L^p(G), L^q(G)$ norms respectively we may assume that $f, g$ vanish on $E$, so that the above estimates now hold for all $x$.  In particular, $f$ and $g$ are supported on the same set $H$.  Solving for $f, g$, we conclude that there exist constants $C_1, C_2 > 0$ such that
$$ f(x) = (1+o(1)) C_1 1_H(x); \quad g(x) = (1+o(1)) C_2 1_H(x)$$
for all $x \in G$.  Using the normalisation of $f, g$, and modifying $f, g$ by $o(1)$ in the $L^p(G), L^q(G)$ norms respectively, we may assume that $H$ has positive measure and
$$ f = \mu(H)^{-1/p} 1_H; \quad g = \mu(H)^{-1/q} 1_{H}.$$

Inserting this back into \eqref{goo}, we have after some algebra that
$$
 \int_G \mu(H \cap (H+h))^{s/r}\ d\mu(h) \geq (1-o(1)) \mu(H)^{s/r + 1}.
$$
On the other hand, from Fubini's theorem we have
\begin{equation}\label{fub}
 \int_G \mu(H \cap (H+h))\ d\mu(h) = (1-o(1)) \mu(H)^2
\end{equation}
and thus
$$
 \int_G \mu(H \cap (H+h)) (\mu(H)^{s/r - 1} - \mu(H \cap (H+h))^{s/r-1} ) \ d\mu(h) = o(\mu(H)^{s/r+1}).$$
Applying Lemma \ref{mes}, we can find a measurable set $K$ such that
$$ \int_{G \backslash K} \mu(H \cap (H+h))\ d\mu(h) = o(\mu(H)^2)$$
and such that
$$ \mu(H \cap (H+h)) (\mu(H)^{s/r - 1} - \mu(H \cap (H+h))^{s/r-1} ) = o(\mu(H)^{s/r - 1} \mu(H \cap (H+h)))$$
for all $h \in K$.  Discarding those $h$ for which $\mu(H \cap (H+h)) = 0$, we may thus assume that
$$ \mu(H \cap (H+h)) = (1-o(1)) \mu(H)$$
for all $h \in K$.  By continuity of the function $h \mapsto \mu(H \cap (H+h))$, we may take $K$ to be compact.  Integrating the above bound in $h$ and using \eqref{fub}, we have
$$ \mu(K) = (1+o(1)) \mu(H).$$
Also, if $h, h' \in K$, then
$$ \mu(H \cap (H+h)), \mu((H+h) \cap (H+h-h')) = (1-o(1)) \mu(H)$$
and hence by the triangle inequality
$$ \mu(H \cap (H+h-h')) = (1-o(1)) \mu(H).$$
Thus, for every $k$ in the difference set $K-K$ (which is also compact), one has
$$ \mu(H \cap (H+k)) = (1-o(1)) \mu(H).$$
Integrating this in $k$ and using \eqref{fub}, we conclude that
$$ \mu(K-K) = (1 + o(1)) \mu(H) = (1+o(1)) \mu(K).$$
By Lemma \ref{ise}, we conclude that $H_0 := K-K$ is a compact open subgroup of $G$.  By the above discussion, we have $\mu(H \cap (H+h)) = (1-o(1)) \mu(H)$ for all $h \in H_0$.  Averaging in $h$, we conclude that $1_H$ and $\frac{1}{\mu(H_0)} 1_{H_0} * 1_H$ differ by at most $o(\mu(H))$ in $L^1$ norm.  As the latter function is constant along cosets of $H_0$, and the former function takes values $0$ and $1$ and has a total $L^1$ norm of $\mu(H) = (1+o(1)) \mu(H_0)$, both functions must differ from the indicator function of a single coset by $o(\mu(H_0))$ in $L^1$ norm.  In other words, there is a coset $x+H_0$ of $H_0$ such that $\mu(H \cap (x+H_0)) = (1-o(1)) \mu(H_0)$, and thus $\mu(H \backslash (x+H_0)) = o(1)$.  Thus, after modifying $f, g$ by $o(1)$ in $L^p$ and $L^q$ norm respectively, we may assume that $H = x+H_0$, and the claim follows.
\end{proof}

\section{$L^p$ near-extremisers on groups}\label{p-group-sec}

We can now quickly prove Theorem \ref{p-group}.  Again, we only establish the claim for $\eps > 0$.  Fix $k \geq 2$ and $G$.  It suffices to show that if $f \in L^{p_k}(G)$ is such that $\|f\|_{L^{p_k}(G)} \leq 1+o(1)$ and $\|f\|_{U^k(G)} \geq 1-o(1)$, then there exists a coset $H = x_0+H_0$ and a polynomial $P: H_0 \to \R/\Z$ of degree $\leq k-1$ such that $\|f-\mu(H)^{-1/p_k} 1_H e(P(\cdot-x_0)) \|_{L^{p_k}(G)} = o(1)$.

From \eqref{fk}, \eqref{ukg2} one has
$$ \int_G \|(T^h f) \overline{f} \|_{L^{p_{k-1}}(G)}^{2^{k-1}}\ d\mu(h) \geq 1-o(1).$$
Applying Proposition \ref{iyi} we conclude that $\| |f| - \mu(H)^{-1/p_k} 1_{x_0+H} \|_{L^{2^k/(k+1)}(G)} = o(1)$ for some coset $x_0+H$ of a subgroup $H$ of $G$.  By translating, we may set $x_0=0$; by modifying $f$ by $o(1)$ in $L^{2^k/(k+1)}(G)$ norm, we may assume that
$$ |f| = \mu(H)^{-1/p_k} 1_{H}.$$
By dividing the Haar measure $\mu$ by $\mu(H)$ (and multiplying $f$ by $\mu(H)^{1/p_k}$) we may normalise so that $\mu(H)=1$, thus we now have $|f|=1_H$.  At this point we may restrict $f$ to $H$, and the claim follows from Theorem \ref{infty-group}.

\section{Ergodic theory analogues}

We can now adapt Proposition \ref{young} and Proposition \ref{iyi} to ergodic systems:

\begin{proposition}[Young's inequality for ergodic systems]\label{young-ergodic}  Let $X = (X, {\mathcal X}, \mu, T)$ be an ergodic system, and let $0 < r < p,q < s < \infty$ be such that $\frac{1}{r}+\frac{1}{s}=\frac{1}{p}+\frac{1}{q}$.  Then for any $f \in L^p(X)$ and $g \in L^q(X)$, one has
\begin{equation}\label{lhs}
 \limsup_{H \to\infty} (\E_{h \in [H]} \| (T^h f) g \|_{L^r(X)}^s)^{1/s} \leq \|f\|_{L^p(X)} \|g\|_{L^q(X)}.
\end{equation}
\end{proposition}

\begin{proof} Repeating the proof of Proposition \ref{young}, one has
$$ \|(T^h f) g\|_{L^r(X)}^s \leq \|f\|_{L^p(X)}^{s-p} \|g\|_{L^q(X)}^{s-q} \int_X |T^h f(x)|^p |g(x)|^q\ d\mu(x)$$
for any $h \in \Z$.  On the other hand, from the ergodic theorem one has
$$
\lim_{H \to \infty} \E_{h \in [H]} \int_G |T^h f(x)|^p |g(x)|^q\ d\mu(x) = \|f\|_{L^p(G)}^p \|g\|_{L^q(G)}^q,
$$
and the claim follows.
\end{proof}

\begin{proposition}[Inverse Young inequality]\label{iyi-erg} Let the notation and hypotheses be as in Proposition \ref{young-ergodic}.  Assume the normalisation
\begin{equation}\label{normpa}
 \|f\|_{L^p(X)} = \|g\|_{L^q(X)} = 1
\end{equation}
and assume that
\begin{equation}\label{tfga}
\limsup_{H \to \infty} \left( \E_{h \in [H]} \int_X \|(T^h f) g\|_{L^r(X)}^s\right)^{1/s} \geq 1-\eps
\end{equation}
for some $\eps > 0$.  Then there exists a coset $H$ of $X$ and an integer $h$ such that
$$ \||f|- \mu(H)^{-1/p} 1_{H}\|_{L^p(X)} = o_{p,q,r,s}(1); \quad \||g|- \mu(H)^{-1/q} 1_{T^h H} \|_{L^q(G)} = o_{p,q,r,s}(1).$$
\end{proposition}

\begin{proof}  As before, we abbreviate $o_{p,q,r,s}(1)$ as $o(1)$, assume $\eps$ is small, and relax the right-hand sides of \eqref{normpa}, \eqref{tfga} to $1 + o(1)$ and $1-o(1)$ respectively.  Repeating the proof of Proposition \ref{iyi}, we may reduce to the case where
$$ f = \mu(H)^{-1/p} 1_H; \quad g = \mu(H)^{-1/q} 1_{H}$$
for some measurable subset $H$ of $X$ of positive measure.  To finish the proof, it suffices to show that $H$ differs from a coset by a set of measure $o(\mu(H))$.

Arguing as in the proof of Proposition \ref{iyi}, we have
$$
\limsup_{N \to \infty} \E_{h \in [N]}\mu(H \cap T^h H)^{s/r} \geq (1-o(1)) \mu(H)^{s/r + 1},$$
while from the ergodic theorem one has
$$
\lim_{N \to \infty} \E_{h \in [N]}\mu(H \cap T^h H) = \mu(H)^2.$$
Also, $\mu(H \cap T^h H))$ is trivially bounded above by $\mu(H)$.  We conclude that
$$
\limsup_{N \to \infty} \E_{h \in [N]}\mu(H \cap T^h H)^2 \geq (1-o(1)) \mu(H)^3,$$
which by Definition \ref{semidef} implies that
\begin{equation}\label{u2}
 \|1_H\|_{U^2(X)} \geq (1-o(1)) \mu(H)^{3/4}.
\end{equation}
To use this, we introduce the theory of the \emph{Kronecker factor}.
Let ${\mathcal Z}_1$ be the sub-$\sigma$-algebra of ${\mathcal X}$
generated by the eigenfunctions of $X$ (or equivalently, by the pure
point spectrum of $T$).  As is well known (see e.g. \cite{host-kra}),
the \emph{Kronecker factor} $(X, {\mathcal Z}_1,
\mu\downharpoonright_{{\mathcal Z}_1}, T)$ is equivalent to an ergodic
rotation $Z_1 = (Z_1, {\mathcal Z}_1, \mu_1, T_1)$, i.e. a compact
abelian group $Z_1 = (Z_1,+)$ with Haar measure $\mu_1$ and an ergodic
shift $T_1: z \mapsto z+\alpha$ for some $\alpha \in Z_1$.  In
\cite[Lemma 4.3]{host-kra}, it is shown that $\|f\|_{U^2(X)}=0$
whenever the function $f \in L^\infty(X)$ is such that the conditional expectation $\E(f|Z_1)$ vanishes.  In particular,
$$ \|1_H - \E(1_H|Z_1)\|_{U^2(X)} = 0$$
and thus by the triangle inequality
$$ \|1_H \|_{U^2(X)} = \| \E(1_H|Z_1)\|_{U^2(X)}.$$
Applying \eqref{u2}, \eqref{ukp2} we conclude that
$$ \| \E(1_H|Z_1)\|_{L^{4/3}(X)} \geq (1-o(1)) \mu(H)^{3/4}.$$
Since
$$ \| \E(1_H|Z_1)\|_{L^1(X)} = \| 1_H \|_{L^1(X)} = \mu(H)$$
we conclude from H\"older's inequality that
$$  \| \E(1_H|Z_1)\|_{L^2(X)}^2 \geq (1-o(1)) \mu(H).$$
Since
$$  \| 1_H \|_{L^2(X)}^2 = \mu(H)$$
we conclude from Pythagoras' theorem that
$$ \| 1_H - \E(1_H|Z_1) \|_{L^2(X)}^2 = o( \mu(H) ).$$
As a consequence there is a $Z_1$-measurable subset $H'$ that differs from $H$ by a set of measure $o(\mu(H))$.  Thus, without loss of generality, we may assume that $H=H'$, i.e. that $H$ is $Z_1$-measurable.  We may thus restrict $X$ to the Kronecker factor $Z_1$ and assume without loss of generality that $X=Z_1$.  By \cite[\S 3.2]{host-kra}, the Gowers-Host-Kra seminorm $U^2(Z_1)$ on the Kronecker factor $Z_1$ coincides with the Gowers uniformity norm $U^2(Z_1)$ on the compact abelian group $Z_1$.  Applying Theorem \ref{p-group}, we see that $H$ differs from a coset by a set of measure $o(\mu(H))$, as required.
\end{proof}

\section{$L^p$ near-extremisers on systems}\label{p-system-sec}

Now we establish Theorem \ref{p-system}.
Again, we only establish the claim for $\eps > 0$.  Fix $k \geq 2$ and $X$.  It suffices to show that if $f \in L^{p_k}(X)$ is such that $\|f\|_{L^{p_k}(X)} \leq 1+o(1)$ and $\|f\|_{U^k(X)} \geq 1-o(1)$, then there exists a coset $H$ and a polynomial $P: H \to \R/\Z$ of degree $\leq k-1$ such that $\|f-\mu(H)^{-1/p_k} 1_H e(P(\cdot-x_0)) \|_{L^{p_k}(X)} = o(1)$.

From \eqref{usx}, \eqref{ukp} one has
$$ \lim_{H \to \infty} \E_{h \in [H]} \|(T^h f) \overline{f} \|_{L^{p_{k-1}}(G)}^{2^{k-1}} \geq 1-o(1).$$
Applying Proposition \ref{iyi-erg} we conclude that $\| |f| - \mu(H)^{-1/p_k} 1_{H} \|_{L^{2^k/(k+1)}(X)} = o(1)$ for some coset $H$ of $X$. By modifying $f$ by $o(1)$ in $L^{2^k/(k+1)}(X)$ norm, we may assume that
$$ |f| = \mu(H)^{-1/p_k} 1_{H}.$$
If we then replace the system $X$ with the normalisation of the coset $H$ (as defined in Definition \ref{cosets}), and set $\tilde f: H \to \C$ be the function $\tilde f(x) := \mu(H)^{1/p_k} f(x)$ for $x \in H$, we see (as in the introduction) that
$$ \|f\|_{L^{p_k}(X)} = \| \tilde f\|_{L^{p_k}(H)}$$
and
$$ \|f\|_{U^k(X)} = \| \tilde f\|_{U^k(H)}.$$
We may thus pass from $X$ to $H$, and the claim now follows from Theorem \ref{infty-ergodic}.

\begin{remark}\label{coset}  By combining Theorem \ref{p-system} with Theorem \ref{p-group} we see that a measurable subset $H$ of an ergodic system $X$ is a coset of $X$ (in the ergodic theory sense) if and only if $H$ is a coset of the Kronecker factor $Z_1$ (in the group theory sense), modulo null sets.  Indeed, the ``if'' part is easy.  To see the ``only if'' part, observe that if $H$ is a coset then $1_H$ can be decomposed as a linear combination of eigenfunctions $\sum_{j=1}^m e(jk/m) 1_{T^j H}$ of $T$, and so $H$ is measurable with respect to the Kronecker factor, which is a compact abelian group.  By Remark \ref{uni-kra}, the Gowers and Gowers-Host-Kra norms of $1_H$ then coincide, and the claim follows from Theorem \ref{p-system} and Theorem \ref{p-group}.  Alternatively, one can observe from the density of the orbit $\{ n \alpha: n \in \Z \}$ of the ergodic shift $\alpha \in Z_1$ in $Z_1$ that for every $\theta \in Z_1$, $H+\theta$ is either equal to $H$ or is disjo
 int from $H$ modulo null sets, which then implies that $H$ is a coset of $Z_1$ in the group-theoretic sense up to null sets.
\end{remark}

\section{The Gowers norm on Euclidean spaces}\label{syg-sec}

We now prove Theorem \ref{syg}.  Our main tool is the following refinement of Proposition \ref{young} in the Euclidean case, due to
Beckner \cite{beckner} and Brascamp and Lieb \cite{brascamp}:

\begin{theorem}[Sharp Young's inequality]\label{young-beck}\cite{beckner}, \cite{brascamp}  Let $n \geq 1$ be an integer, and let $0 < r < p,q < s < \infty$ be such that $\frac{1}{r}+\frac{1}{s}=\frac{1}{p}+\frac{1}{q}$.  Then for any $f \in L^p(\R^n)$ and $g \in L^q(\R^n)$, one has
\begin{equation}\label{young-ineq-sharp}
 \left(\int_{\R^n} \|(T^h f) g\|_{L^r(\R^n)}^s\ dh \right)^{1/s}\leq
 (A_{p/r} A_{q/r} / A_{s/r})^{n/r}
 \|f\|_{L^p(\R^n)} \|g\|_{L^q(\R^n)}
\end{equation}
where $A_p := \left(\frac{p^{1/p}}{(p')^{1/p'}}\right)^{1/2}$ and $1/p + 1/p' = 1$.  If $f,g$ are non-negative, then equality occurs if and only if one has $f = c_0 e^{-p (x-x_0) \cdot M(x-x_0)}$ and $g = c_1 e^{-q(x-x_1) \cdot M(x-x_1)}$ for some $c_0, c_1 \in \C$, $x_0, x_1 \in \R^n$, and positive definite $M$.
\end{theorem}

\begin{remark}  The results of Beckner and Brascamp-Lieb directly handle the $r=1$ case of this theorem, but the general case then follows by the substitution $f \mapsto |f|^r$, $g \mapsto |g|^r$, $s \mapsto s/r$, $p \mapsto p/r$, $q \mapsto q/r$.
\end{remark}

Specialising this inequality to the exponents \eqref{pqk} and $g=\overline{f}$ for some $k \geq 2$, we see that
$$
 \left(\int_{\R^n} \|(T^h f) \overline{f}\|_{L^{p_{k-1}}(\R^n)}^{2^{k-1}}\ dh \right)^{1/2^{k-1}}\leq
 \left(\frac{A_{2k/(k+1)}^2}{A_{k}}\right)^{nk/2^{k-1}} \|f\|_{L^{p_k}(\R^n)}^2.$$
A calculation shows that
$$ \frac{A_{2k/(k+1)}^2}{A_{k}} = \frac{2^{1/k} k^{1/2}}{(k+1)^{(k+1)/2k}}$$
and thus
$$ C_k = \left(\frac{A_{2k/(k+1)}^2}{A_{k}}\right)^{k/2^k} C_{k-1}^{1/2}.$$
Because of this, \eqref{ud-sharp} follows from \eqref{fk}, \eqref{young-ineq-sharp}, and induction.

From Theorem \ref{young-beck} we see from the above argument that equality holds when $f$ takes a gaussian form
$$ f(x) = c e^{-(x-x_0) \cdot M (x-x_0)} $$
for some $c \in \C$, $x_0 \in \R^n$ and a positive-definite $n \times n$ matrix $M$.  From \eqref{pepg} we see that the same holds for functions $f$ of the form \eqref{fphi}.

\begin{remark}  One can of course also verify equality in the Gaussian case by direct computation.  To illustrate this we consider the model case when $n=1$ and $f(x) = e^{-\pi |x|^2}$. The right-hand side of \eqref{ud-sharp} is easily seen to equal
$$ (2^d/(d+1))^{-(d+1)/2^{d+1}}$$
while the left-hand side is equal to
$$ |\det M|^{-1/2}$$
where $M$ is the $d+1 \times d+1$ matrix for the quadratic form
$$ Q(x,h_1,\ldots,h_d) := \sum_{\omega_1,\ldots,\omega_d \in \{0,1\}} |x + \omega_1 h_1 + \ldots + \omega_d h_d|^2.$$
A short computation shows that
$$
M = \begin{pmatrix}
2^d & 2^{d-1} & 2^{d-1} & \ldots & 2^{d-1} \\
2^{d-1} & 2^{d-1} & 2^{d-2} & \ldots & 2^{d-2} \\
2^{d-1} & 2^{d-2} & 2^{d-1} & \ldots & 2^{d-2} \\
\vdots & \vdots  & \vdots & \ddots & \vdots \\
2^{d-1} & 2^{d-2} & 2^{d-2} & \ldots & 2^{d-1}
\end{pmatrix}.
$$
Using Gaussian elimination (subtracting half of the first row from the remaining rows) one sees that
$$ \det M = 2^d 2^{(d-2)d} = 2^{d(d-1)}$$
and the claim then follows after another short computation.
\end{remark}

Now we consider the converse problem.
Suppose that equality holds in \eqref{ud-sharp}.  Since
$$  \|f\|_{U^k(\R^n)}  \leq  \||f|\|_{U^k(\R^n)}$$
equality must also hold for $|f|$.  By Theorem \ref{young-beck} and the above argument, we thus see that
$$ |f(x)| = c e^{-(x-x_0) \cdot M (x-x_0)} $$
for some $c \geq 0$, $x_0 \in \R^n$ and a positive-definite $n \times n$ matrix $M$.  We may of course assume that $c$ is non-zero, in which case we can normalise $c=1$.  By translation symmetry we can normalise $x=0$, and by a linear change of variables we may assume that $M$ is the identity, thus
$$ f(x) = e^{-|x|^2} e(P(x))$$
for some $P: \R^n \to \R/\Z$.  Since $e^{-|x|^2}$ is always non-zero, and since $f$ and $|f|$ must have the same $U^k(\R^n)$ norm, we see from \eqref{fug} that
$$ \Delta_{h_1} \ldots \Delta_{h_k} P(x) = 0$$
for almost every $h_1,\ldots,h_k, x \in \R^d$, and so $P$ is a polynomial of degree at most $k-1$.  The claim follows.

\begin{remark} The same method also gives sharp Young-type inequalities for the Gowers inner product
$$\langle (f_\omega)_{\omega \in \{0,1\}^k} \rangle_{U^k(\R^n)} := \int_{\R^n} \ldots \int_{\R^n} \prod_{\omega_1,\ldots,\omega_k \in \{0,1\}^k} {\mathcal C}^{\omega_1+\ldots+\omega_k} f_{\omega_1,\ldots,\omega_k}( x + \omega_1 h_1 + \ldots + \omega_k h_k );$$
we omit the details.
\end{remark}

\begin{remark}  In the case $k=3$, the above inequality becomes
$$ \|f\|_{U^3(\R^n)} \leq 2^{-n/8} \|f\|_{L^2(\R^n)}.$$
It is worth noting the large number of invariances that this inequality enjoys.  We have already observed the invariance with respect to linear changes of variable, translation, and multiplication by quadratic phases $e(\phi(x)) = e(x \cdot Mx + \xi \cdot x + \theta)$.  But a calculation also shows that the $U^3$ norm (just like the $L^2$ norm) is preserved by the Fourier transform operation $f \mapsto \hat f$, where
$$ \hat f(\xi) := \int_{\R^n} f(x) e^{-2\pi i x \cdot \xi}\ dx.$$
As a consequence, the $U^3(\R^n)$ norm is in fact preserved by the entire metaplectic representation\footnote{The \emph{metaplectic group} is a double cover of the symplectic group on $\R^{2n}$; its (unitary) action on $\R^n$ is known as the \emph{metaplectic representation} is generated by modulations $f(x) \mapsto f(x) e(P(x))$ by homogeneous quadratic polynomials $P$, linear changes of variable $f(x) \mapsto |\det(A)|^{1/2} f(Ax)$, and the Fourier transform.  See e.g. \cite{stein} for further discussion.} on $\R^n$.  Further Fourier-analytic symmetries of this type will be implicitly exploited in the proof of Theorem \ref{u3-system} in the next section.
\end{remark}

\section{Threshold for $U^3$ on systems}\label{u3-sys}

We now prove Theorem \ref{u3-system}.   We begin with the first claim; the claim that $2^{-1/8}$ is best possible will be deferred to the end of this section.

For any ergodic system $X$, define the \emph{Abramov factor} $A_2$ to be the factor $A_2 = (X, {\mathcal A}_2, \mu\downharpoonright_{{\mathcal A}_2}, T)$, where ${\mathcal A}_2$ is the sub-$\sigma$-algebra of ${\mathcal B}$ generated by the polynomials $P$ of degree at most $2$.  It will suffice to show that
\begin{equation}\label{u3-abr}
 \|f\|_{U^3(X)} \leq 2^{-1/8} \|f\|_{L^2(X)}
\end{equation}
whenever $f$ is orthogonal to the Abramov factor (thus $\E(f|A_2)=0$).  By a limiting argument (using \eqref{ukp}) we may assume that $f \in L^\infty(X)$.

We now use the machinery of Host and Kra \cite{host-kra} to reduce to the case when $X$ is a $2$-step nilsystem:

\begin{definition}[$2$-step nilsystem]  A \emph{$2$-step nilsystem} is an ergodic system of the form $(G/\Gamma, {\mathcal B}, \mu, T)$, where
$G$ is a $2$-step finite-dimensional nilpotent Lie group, $\Gamma$ is a discrete cocompact subgroup of $G$, $X$ is equipped with the Haar probability measure and the Borel $\sigma$-algebra, and the shift $T$ is given by $Tx = \tau x$ for some $\tau \in G$.
\end{definition}

Recall that in \cite{host-kra}, a factor $Z_2$ of $X$ was constructed with the property that
\begin{equation}\label{fux}
 \|f\|_{U^3(X)} = 0
\end{equation}
whenever $f$ was orthogonal to $Z_2$ (see \cite[Lemma 4.3]{host-kra}), and such that $Z_2$ was a system of order $2$ in the notation of \cite{hk01, host-kra}; see \cite[Proposition 4.11]{host-kra}.  Indeed, $Z_2$ was the maximal factor of order $2$.  It is easy to see that the Abramov factor $A_2$ is a system of order $2$, and so is necessarily a subfactor of $Z_2$.  From \eqref{fux} and the triangle inequality we see that
$$ \|f\|_{U^3(X)} = \|\E(f|Z_2)\|_{U^3(X)}$$
for all $f \in L^2(X)$, and so without loss of generality (replacing $f$ with $\E(f|Z_2)$ if necessary) we may assume that $f$ is $Z_2$-measurable.  In particular, we may assume without loss of generality that $X$ is of order $2$.  Applying \cite[Theorem 10.1]{host-kra}, we conclude that $X$ is an inverse limit of $2$-step nilsystems.  Applying a limiting argument (noting that if $Y$ is a factor of $X$, then the Abramov factor of $Y$ is a factor of the Abramov factor of $X$), we may thus assume without loss of generality that $X$ is a $2$-step nilsystem, $X = G/\Gamma$.  (Note that these reductions do not destroy the total ergodicity of the system.)

Next, we can place the ergodic nilsystem $G/\Gamma$ in a standard form.  As discussed in \cite[\S 4.1, 4.3]{host-kra-2step}, we may assume without loss of generality that the nilsystem obeys the following properties:

\begin{itemize}
\item $G/\Gamma$ is minimal (i.e. every orbit is dense) and uniquely ergodic.
\item The commutator group $G_2 := [G,G]$ is a torus.
\item The group $\Gamma$ is abelian and has trivial intersection with the centre of $G$ (and in particular with $G_2$).
\item $G$ is spanned by the connected component $G^\circ$ of the identity and by $\tau$.
\item The Kronecker factor $Z_1$ is isomorphic to the quotient system $G/(G_2 \Gamma)$, with the obvious factor map.
\end{itemize}

In particular, the torus $G_2$ acts freely on $G/\Gamma$.  Because of this, every function $f \in L^2(G/\Gamma)$ has a Fourier decomposition
\begin{equation}\label{fax}
f = \sum_{\xi \in \hat G_2} f_\xi
\end{equation}
where for each $\xi$ in the Pontryagin dual $\hat G_2$, $f_\xi \in L^2(G/\Gamma)$ is the function
\begin{equation}\label{foxy}
 f_\xi(x) := \int_{G_2} e(-\xi(g_2)) f(g_2 x) d\mu_{G_2}(g_2)
\end{equation}
(with $\mu_{G_2}$ being the Haar probability measure on the torus $G_2$) and the series is unconditionally convergent in $L^2(G/\Gamma)$, with the $f_\xi$ being orthogonal and obeying the Plancherel identity
$$ \|f\|_{L^2(G/\Gamma)}^2 = \sum_{\xi \in \hat G_2} \| f_\xi \|_{L^2(G/\Gamma)}^2.$$
Observe that each $f_\xi$ is an eigenfunction of $G_2$ action, in that
\begin{equation}\label{fxi}
f_\xi(g_2 x) = e(\xi(g_2)) f_\xi(x)
\end{equation}
for all $x \in G/\Gamma$ and $g_2 \in G_2$.

The $U^3(G/\Gamma)$ norm can be rewritten explicitly as an integral, as follows.  Given a $2$-step nilpotent group $G$, define the \emph{Host-Kra group} $\HK^3(G)$ (also known as $G^{[3]}_2$ or $\HP^3(G)$) to be the subgroup of $G^{\{0,1\}^3}$ given by the octuples of the form
$$ \left( \prod_{\omega' \in \{0,1\}^3: \omega'_i \leq \omega_i \hbox{ for all } i=1,2,3} g_{\omega'} \right)_{\omega \in \{0,1\}^3}$$
where $g_{\omega'} \in G_{|\omega'|}$ for all $\omega' \in \{0,1\}^3$, with the convention that $G_0 = G_1 := G$ and $G_3$ is trivial, and the product is ordered in some arbitrary fashion (e.g. lexicographical ordering on $\{0,1\}^3$ will suffice).  One can check that this is indeed a group; see \cite{green-tao-linearprimes}. For a $2$-step nilmanifold, one can show that $\HK^3(\Gamma)$ is a discrete cocompact subgroup of the $2$-step nilpotent Lie group $\HK^3(G)$ (see e.g. \cite[\S B.2]{host-kra-inf}), and thus the quotient space $\HK^3(G)/\HK^3(\Gamma) \subset (G/\Gamma)^{\{0,1\}^3}$ is a nilmanifold with a Haar measure $\mu_{\HK^3(G)/\HK^3(\Gamma)}$.  The $U^3(G/\Gamma)$ norm can then be expressed by the formula
\begin{equation}\label{f8}
\|f\|_{U^3(G/\Gamma)}^8 = \int_{\HK^3(G)/\HK^3(\Gamma)} \prod_{\omega \in \{0,1\}^3} {\mathcal C}^{|\omega|} f( x_\omega )\ d\mu_{\HK^3(G)/\HK^3(\Gamma)}(x)
\end{equation}
where $x = (x_\omega)_{\omega \in \{0,1\}^3}$, and ${\mathcal C}: z \mapsto \overline{z}$ is the complex conjugation map; see \cite[\S B.3-B.5]{host-kra-inf} (and also \cite{host-kra}).

The Fourier decomposition reacts well to both the Abramov factor and the $U^3$ seminorm:

\begin{lemma}  Let $f \in L^\infty(G/\Gamma)$, and let $f_\xi$ be the Fourier components of $f$.
\begin{itemize}
\item $f$ is orthogonal to $A_2$ if and only if each $f_\xi$ is orthogonal to $A_2$.
\item One has $\|f\|_{U^3(G/\Gamma)}^8 = \sum_{\xi \in \hat G_2} \|f_\xi\|_{U^3(G/\Gamma)}^8$.
\end{itemize}
\end{lemma}

\begin{proof} We prove the first claim.  The ``if'' part is trivial from the expansion \eqref{fax}, which is unconditionally convergent in $L^2(G/\Gamma)$.  To show the ``only if'' part, one observes that the action of $G_2$ commutes with $T$ and thus preserves the property of being a polynomial of degree $\leq 2$, and the claim then follows from \eqref{foxy}.

Now we establish the second claim.  By the unconditional convergence of the Fourier expansion $f = \sum_\xi f_\xi$ in $L^2(G/\Gamma)$ and \eqref{ukp}, we may assume without loss of generality that all but finitely many of the $f_\xi$ vanish.  By \eqref{f8}, the left-hand side expands as
$$
\sum_{(\xi_\omega)_{\omega \in \{0,1\}^3} \in (\hat G_2)^{\{0,1\}^3}} \int_{\HK^3(G)/\HK^3(\Gamma)} \prod_{\omega \in \{0,1\}^3} {\mathcal C}^{|\omega|} f_{\xi_\omega}( x_\omega )\ d\mu_{\HK^3(G)/\HK^3(\Gamma)}(x)$$
while the right-hand side expands as
$$
\sum_{\xi \in \hat G_2} \int_{\HK^3(G)/\HK^3(\Gamma)} \prod_{\omega \in \{0,1\}^3} {\mathcal C}^{|\omega|} f_{\xi}( x_\omega )\ d\mu_{\HK^3(G)/\HK^3(\Gamma)}(x).$$
It thus suffices to show that the expression
\begin{equation}\label{hg}
 \int_{\HK^3(G)/\HK^3(\Gamma)} \prod_{\omega \in \{0,1\}^3} {\mathcal C}^{|\omega|} f_{\xi_\omega}( x_\omega )\ d\mu_{\HK^3(G)/\HK^3(\Gamma)}(x)
\end{equation}
vanishes whenever the $\xi_{\omega}$ are not all equal to each other.

Suppose now that $\omega, \omega' \in \{0,1\}^3$ are such that $\xi_\omega \neq \xi_{\omega'}$, then there exists $g_2 \in G_2$ such $\xi_\omega(g_2) \neq \xi_{\omega'}(g_2)$.  Call the integrand in \eqref{hg} $F$.  Then from \eqref{fxi} we have the eigenfunction equation
$$ F( (g_{\omega''})_{\omega'' \in \{0,1\}^3} x ) = e( \xi_\omega(g_2) - \xi_{\omega'}(g_2) ) F(x)$$
for all $x \in \HK^3(G)/\HK^3(\Gamma)$, where $g_{\omega''}$ is equal to $g_2$ when $\omega'' = \omega$, $g_2^{-1}$ when $\omega'' = \omega'$, and the identity otherwise.  On the other hand, the element $(g_{\omega''})_{\omega'' \in \{0,1\}^3}$ is easily verified to lie in $\HK^3(G)$, and thus leaves the Haar measure $\mu_{\HK^3(G)/\HK^3(\Gamma)}$ invariant.  As the factor $e( \xi_\omega(g_2) - \xi_{\omega'}(g_2) )$ is not equal to $1$, the integral in \eqref{hg} must vanish, and the claim follows.
\end{proof}

In view of the above lemma, we see that in order to prove the inequality \eqref{u3-abr} for $f$, it suffices to do so for each $f_\xi$.  In particular, we may assume without loss of generality that $f$ obeys \eqref{fxi} for some $\xi \in \hat G_2$.

Next, we reduce to the case when $G$ is connected, as follows.  Suppose that $G$ is not connected.
In this case, $G/G^\circ\Gamma$ is a discrete compact group (i.e. a finite group) generated by $\tau$.  As $\tau$ is totally ergodic, this group must be trivial.  Thus we have $\tau = g_0 \gamma$ for some $g_0 \in G^\circ$ and $\gamma \in \Gamma$.   We have
$$ \tau x = g_0 x [\gamma,x]$$
and thus
\begin{equation}\label{tauj}
\tau^j x = g_0^j x [\gamma,x]^j [\gamma,g_0]^{\frac{j(j-1)}{2}}.
\end{equation}
In particular
\begin{equation}\label{ftauj}
f(\tau^j x) = f(g_0^j x) e( j \xi( [\gamma,x] ) + \frac{j(j-1)}{2} \xi([\gamma,g_0]) ).
\end{equation}
The group elements $\tau, g_0$ have the same projection to the
abelianisation $G/G_2\Gamma$.  As $\tau$ is totally ergodic, we
conclude from a theorem of Leibman \cite{leibman} that $g_0$ is also totally ergodic.   Thus we can build a nilsystem with nilmanifold $G^\circ/(G^\circ \cap \Gamma) \equiv G/\Gamma$ and shift given by $g_0$, which has the same Haar measure as the original nilsystem $G/\Gamma$.  From \eqref{ftauj} and a short calculation we see that $f$ has the same $U^3$ norm with respect to the original nilsystem $G/\Gamma$ as it does with the new system $G^\circ/(G^\circ \cap \Gamma)$; also, from \eqref{tauj}, every polynomial of degree $\leq 2$ in the original nilsystem remains a polynomial of degree $\leq 2$ in the new system and vice versa.  Finally, the property \eqref{fxi} is retained (possibly after restricting $\xi$ to a smaller commutator subgroup $[G^\circ,G^\circ]$ if necessary) after passing from the old nilsystem to the new one.  From this discussion we see that we may assume without loss of generality that $G=G^\circ$, so that $G$ is connected.

If $\xi$ is zero, then $f$ is now $G_2$-invariant, and thus $Z_1$-measurable.  But as $Z_1$ is contained in the Abramov factor $A_2$, and $f$ is orthogonal to $A_2$, this implies that $f$ is trivial, in which case \eqref{u3-abr} is also trivial.  So we may assume that $\xi$ is nonzero.  In this case, $f$ is invariant with respect to the subtorus $\xi^\perp := \{ g_2 \in G_2: \xi(g_2) = 0 \}$ of $G_2$; by quotienting out by this subtorus, we may assume that $\xi^\perp$ is trivial, so that $G_2$ is now isomorphic to the unit circle $\R/\Z$.  In the notation of \cite{host-kra-2step}, this means that the $2$-step nilmanifold $G/\Gamma$ is \emph{elementary}.

We will no longer need the hypothesis that $f$ is orthogonal to the Abramov factor, as this has been superseded by the non-zero nature of $\xi$ and the connected nature of $G$.  We therefore discard this hypothesis, as this will free us to perform some additional transformations on the nilsystem $G/\Gamma$.

The elementary connected nilmanifolds were classified in \cite[Lemma 12]{host-kra-2step}.  According to that lemma, the nilsystem $G/\Gamma$ is isomorphic to a product $G'/\Gamma' \times (\R/\Z)^m$, where $m \ge 0$, $G'$ is the $2$-step nilpotent group $G' = \R^{2d} \times \R/\Z$ for some $d \geq 1$ with multiplication law
$$ (x,z) (x,z') = (x+x', z+z' + \langle Ax, x' \rangle )$$
where $A$ is a $2d \times 2d$ a matrix with integer entries such that $B := A - A^t$ is nonsingular, and $\Gamma' := \Z^{2d} \times \{0\}$.

\begin{remark}  A model case to keep in mind is when $d=1$, $m=0$, and $A = \begin{pmatrix} 0 & 1 \\ 0 & 0 \end{pmatrix}$.  Indeed, our strategy will essentially be to reduce to this case.
\end{remark}

The next step is to eliminate the torus factor $(\R/\Z)^m$.  For this we use a convenient Fubini-type theorem for the Gowers-Host-Kra norm:

\begin{lemma}[Fubini-type theorem]\label{Fub}  Let $k \geq 1$.  If $X, Y$ are two ergodic systems, and $f \in L^\infty(X \times Y)$, then
$$ \| f \|_{U^k(X \times Y)} \leq \| F \|_{U^k(Y)}$$
where $F \in L^\infty(Y)$ is the function $F(y) := \| f_y \|_{U^k(X)}$, and $f_y \in L^\infty(X)$ is the function $f_y(x) := f(x,y)$.
\end{lemma}

\begin{proof}  From \cite[\S 3]{host-kra}, one has
\begin{equation}\label{fy}
 \| f \|_{U^k(X \times Y)}^{2^k} = \int_{(X \times Y)^{[k]}} \prod_{\omega \in \{0,1\}^k} {\mathcal C}^{|\omega|} f(x_\omega,y_\omega)\ d\mu^{[k]}_{X \times Y}(x,y)
\end{equation}
where the cubic measure space $(X \times Y)^{[k]}, \mu_{X \times Y}^{[k]})$ is defined in \cite[\S 3]{host-kra}.  An inspection of that construction reveals that
$$ \mu^{[k]}_{X \times Y} = \mu^{[k]}_X \times \mu^{[k]}_Y$$
and thus (by Fubini's theorem) we can rewrite the previous expression as
$$ \int_{Y^{[k]}} \int_{X^{[k]}} \prod_{\omega \in \{0,1\}^k} {\mathcal C}^{|\omega|} f_{y_\omega}(x_\omega)\ d\mu^{[k]}_{X}(x) d\mu^{[k]}_Y(y).$$
From the Cauchy-Schwarz-Gowers inequality (see \cite[Lemma 3.9]{host-kra}), we have
$$ \int_{X^{[k]}} \prod_{\omega \in \{0,1\}^k} {\mathcal C}^{|\omega|} f_{y_\omega}(x_\omega)\ d\mu^{[k]}_{X}(x) \leq \prod_{\omega \in \{0,1\}^k} F(y_\omega),$$
and the claim follows by applying \eqref{fy} for $F$.
\end{proof}

From this inequality and the Fubini-Tonelli theorem, we see that to establish \eqref{u3-abr} for $f$, it suffices to do so for the functions $f_y: x \mapsto f(x,y)$ in $G'/\Gamma'$ for all $y \in (\R/\Z)^m$.  Thus we may assume without loss of generality that $m=0$ and $G/\Gamma = (G'/\Gamma')$.

We can now view $f$ as a bounded measurable function $f: \R^{2d} \times (\R/\Z) \to \C$ obeying the periodicity condition
$$ f( x+n, z + \langle Ax, n \rangle) = f(x,z)$$
for all $x \in \R^{2d}, n \in \Z^{2d}, z \in \R/\Z$, as well as the frequency condition
$$ f( x, z + \theta ) = e(\xi \theta) f(x,z)$$
for some non-zero integer $\xi$.  By conjugation symmetry we may take $\xi$ to be positive.  (The reader may wish to keep the simple case $\xi=1$ in mind for a first reading, as some of the more technical complications are avoided in that case.)  We may thus factorise
\begin{equation}\label{twist}
 f(x,z) = F(x) e(\xi z)
\end{equation}
where $F$ now obeys the twisted periodicity condition
$$ F(x+n) = e(- \xi \langle Ax, n \rangle) F(x)$$
for all $x \in \R^{2d}, n \in \Z^{2d}$.

We make the observation that without loss of generality, we may refine the lattice $\Gamma = \Z^{2d} \times \{0\}$ to any sublattice of $\Gamma$, thus lifting the nilsystem to a finite cover.  Indeed, the total ergodicity of the system does not change (as can be seen by Leibman's criterion \cite{leibman} for total ergodicity), and the $L^2$ and $U^3$ norms of $f$ are also seen to be unaffected (renormalising the Haar measure to be a probability measure, of course).  We shall need to exploit this freedom to refine the lattice shortly.

We now transform the system to a Heisenberg normal form.  As observed in \cite[\S 8]{host-kra-2step}, there exists a $2d \times 2d$ nonsingular matrix $\Phi$ with rational entries such that
$$ A-A^t = \Phi^t J \Phi$$
where $J$ is the $2d \times 2d$ matrix
$$ J :=\begin{pmatrix} 0 & I_d \\ -I_d & 0 \end{pmatrix}.$$
As $\Phi$ is rational and nonsingular, one can find a positive integer $q$ such that $\Phi ( q \Z^{2d} )$ is a sublattice of $\Gamma$; we may also assume without loss of generality that $q$ is even.  Passing to this sublattice of $\Gamma'$ and applying the change of variables $(x,z) = (\Phi(x'),z)$, we may now assume without loss of generality that $A-A^t = J$, at the cost of replacing $\Gamma$ with $(q \Z^{2d} \times \{0\})$.  Thus, in particular, we now have
$$ F(x+n) = e(- \xi\langle Ax, n \rangle) F(x)$$
for all $x \in \R^{2d}$ and $n \in q\Z^{2d}$.

We can write
$$ A = \begin{pmatrix} 0 & I_d \\ 0 & 0 \end{pmatrix} + A'$$
where $A'$ is symmetric.  If we then apply the change of variable $(x,z) = (x,\tilde z+\frac{1}{2} \langle A'x, x \rangle)$, this has the effect of replacing $A$ with $\begin{pmatrix} 0 & I_d \\ 0 & 0 \end{pmatrix}$ without actually affecting the nilmanifold or $f$ (here we use that $q$ is even to show that $q\Z^{2d} \times \{0\}$ is unaffected by the change of variables).  Thus we may assume without loss of generality that
$$ A = \begin{pmatrix} 0 & I_d \\ 0 & 0 \end{pmatrix}.$$

By lifting the nilsystem to a finite cover via the change of coordinates $(x,z) = (qx',q^2 z')$ and replacing $\xi$ by $q^2 \xi$, we may assume without loss of generality that $q=1$.  Thus, splitting $x \in \R^{2d}$ as $(x_1,x_2)$ with $x_1,x_2 \in \R^d$, the group $G$ now has law
$$ (x_1,x_2,z) (x'_1,x'_2,z') = (x_1+x'_1,x_2+x'_2,z+z'+ x'_1 x_2)$$
and $F$ now obeys the twisted periodicity condition
$$ F(x_1+n_1,x_2+n_2) = e( - \xi n_1 \cdot x_2 ) F(x_1,x_2)$$
whenever $x_1,x_2 \in \R^d$ and $n_1,n_2 \in \Z^d$.  In particular, $F$ is periodic in the $\{0\} \times \Z^d$ directions, and can thus be viewed as a function on $\R^d \times (\R^d/\Z^d)$.

A fundamental domain for $G/\Gamma$ is given by $[0,1)^d \times [0,1)^d \times (\R/\Z)$, and the Haar measure is given by Lebesgue measure. The $L^2(G/\Gamma)$ norm of $f$ can then be expressed as
\begin{equation}\label{01}
 \|f\|_{L^2(G/\Gamma)} = \left(\int_{[0,1)^d \times (\R^d/\Z^d)} |F(x_1,x_2)|^2\ dx_1 dx_2\right)^{1/2}.
\end{equation}
Now we compute the $U^3$ norm.  A calculation shows that $\HK^3(G)$ consists of tuples of the form
$$( (x_{1,\omega}, x_{2,\omega}, z_\omega) )_{\omega \in \{0,1\}^3}$$
with the property that the $x_{i,\omega} \in \R^d$ depend linearly on $\omega$, thus
\begin{equation}\label{xio}
x_{i,\omega} = x_{i,0} + \sum_{j=1}^3 x_{i,j} \omega_j
\end{equation}
for some $x_{i,0}, x_{i,j} \in \R^d$, and the $z_\omega \in \R/\Z$ obey the condition
$$ \sum_{\omega \in \{0,1\}^3} (-1)^{|\omega|} z_\omega = 0.$$
Indeed, these tuples are easily verified to form a connected Lie group that contains all the generators of $\HK^3(G)$, and its tangent space is contained in the tangent space of $\HK^3(G)$, hence the claim.  The subgroup $\HK^3(\Gamma)$ arises when all the $x_{i,j}$ are integers and the $z_\omega$ vanish.  A fundamental domain for $\HK^3(G)/\HK^3(\Gamma)$ can then be specified by requiring the $x_{i,j}$ to all lie in $[0,1)$.
From this and \eqref{twist}, we see that
\begin{equation}\label{face}
 \|f\|_{U^3(G/\Gamma)}^8 = \int_{([0,1)^d \times (\R/\Z)^d)^3} \prod_{\omega \in \{0,1\}^3} {\mathcal C}^{|\omega|} F( x_{1,\omega}, x_{2,\omega} )\ \prod_{i=0}^3 dx_{1,i} dx_{2,i}
 \end{equation}
where $x_{1,\omega}, x_{2,\omega}$ are defined by \eqref{xio}.

To simplify this expression, we take advantage of the periodicity of $F$ in the $\{0\} \times \Z^d$ to obtain the Fourier decomposition
$$ F(x_1,x_2) = \sum_{k \in \Z^d} F_k(x_1) e( k \cdot x_2 )$$
where $F_k \in L^\infty(\R^d)$ obeys the twisted periodicity condition
\begin{equation}\label{faxi}
 F_k(x_1+n_1) = F_{k+\xi n_1}(x_1)
\end{equation}
for all $x_1 \in \R^d$ and $n_1 \in \Z^d$.  By an approximation argument we may assume that all but finitely many of the $F_k$ vanish on $[0,1]^d$.

We may now expand \eqref{face} as the sum of expressions of the form
\begin{equation}\label{flick}
 \int_{([0,1)^d \times (\R/\Z)^d)^3} \prod_{\omega \in \{0,1\}^3} {\mathcal C}^{|\omega|} F_{k_\omega}( x_{1,\omega} )
e( k_\omega \cdot x_{2,\omega} )\ \prod_{i=0}^3 dx_{1,i} dx_{2,i}
\end{equation}
where $(k_\omega)_{\omega \in \{0,1\}^3}$ are a tuple in $\Z^d$.  Because of the constraint \eqref{xio} with $i=2$, the $x_{2,i}$ integrals vanish unless the $k_\omega$ take the form
\begin{equation}\label{kio}
k_{\omega} = k_{0} + \sum_{j=1}^3 k_{j} \omega_j
\end{equation}
for some $k_0,k_1,k_2,k_3 \in \Z^d$.  Assuming that \eqref{kio} holds, the expression \eqref{flick} simplifies to
$$
 \int_{([0,1)^d)^3} \prod_{\omega \in \{0,1\}^3} {\mathcal C}^{|\omega|} F_{k_\omega}( x_{1,\omega} )\ \prod_{i=0}^3 dx_{1,i}.$$
The next step is to foliate $\Z^d$ into cosets of $\xi \Z^d$ (recall that $\xi$ has been normalised to be a positive integer).  Let $Q$ be the fundamental domain $\{0,1,\ldots,\xi-1\}^d$.  We can then split \eqref{face} as
$$ \sum_{k_0,k_1,k_2,k_3 \in Q} \sum_{n_0,n_1,n_2,n_3 \in \Z^d}
 \int_{([0,1)^d)^3} \prod_{\omega \in \{0,1\}^3} {\mathcal C}^{|\omega|} F_{k_\omega + \xi n_\omega}( x_{1,\omega} )\ \prod_{i=0}^3 dx_{1,i}
$$
where
$$ n_{\omega} := n_{0} + \sum_{j=1}^3 n_{j} \omega_j.$$
Using \eqref{faxi}, this expression can be rewritten as
$$ \sum_{k_0,k_1,k_2,k_3 \in Q} \sum_{n_0,n_1,n_2,n_3 \in \Z^d}
 \int_{([0,1)^d)^3} \prod_{\omega \in \{0,1\}^3} {\mathcal C}^{|\omega|} F_{k_\omega}( x_{1,\omega} + n_\omega )\ \prod_{i=0}^3 dx_{1,i};
$$
viewing the $x_{1,i} \in [0,1)^d$ and $n_{1,i} \in \Z^d$ variables as the fractional and integer parts respectively of a variable $y_i \in \R^d$, we can rewrite this as
$$ \sum_{k_0,k_1,k_2,k_3 \in Q}
 \int_{(\R^d)^3} \prod_{\omega \in \{0,1\}^3} {\mathcal C}^{|\omega|} F_{k_\omega}( y_{\omega} )\ \prod_{i=0}^3 dy_i,
$$
where
$$ y_{\omega} := y_{0} + \sum_{j=1}^3 y_{j} \omega_j.$$
By the Cauchy-Schwarz-Gowers inequality (see e.g. \cite[Lemma 3.9]{host-kra}; the extension to $\R^d$ is routine), we can upper bound this expression as
$$ \sum_{k_0,k_1,k_2,k_3 \in Q}
\prod_{\omega \in \{0,1\}^3} \| F_{k_\omega} \|_{U^3(\R^d)},$$
and then applying Theorem \ref{syg} we can upper bound this in turn by
$$ 2^{-d} \sum_{k_0,k_1,k_2,k_3 \in Q}
\prod_{\omega \in \{0,1\}^3} \| F_{k_\omega} \|_{L^2(\R^d)}.$$
From \eqref{faxi} we see that the function $k \mapsto \|F_k\|_{L^2(\R^d)}$ is periodic with period $\xi \Z^d$, and can thus be viewed as a function $g: \Z^d / \xi \Z^d \to \R$.  We can then rewrite the above expression as
$$ 2^{-d} \sum_{k_0,k_1,k_2,k_3 \in \Z^d/\xi \Z^d} \prod_{\omega \in \{0,1\}^3} g( k_0 + \omega_1 k_1 + \omega_2 k_2 + \omega_3 k_3 )$$
which by \eqref{fug} is just
$$ 2^{-d} \xi^4 \|g\|_{U^3(\Z^d/\xi \Z^d)}^8$$
which by \eqref{ukp} is bounded by
$$ 2^{-d} \xi^4 \|g\|_{L^2(\Z^d/\xi \Z^d)}^8.$$
Using \eqref{faxi} and Fubini's theorem, one can rewrite this as
$$ 2^{-d} \left(\sum_{k \in \Z^d} \| F_k \|_{L^2([0,1)^d)}^2\right)^4$$
which by Plancherel's theorem is equal to
$$ 2^{-d} \| F \|_{L^2([0,1)^d \times (\R/\Z)^d)}^8,$$
and the claim \eqref{u3-abr} then follows from \eqref{01} and the fact that $d \geq 1$.  This concludes the proof of the first part of Theorem \ref{u3-system}.

\subsection{Sharpness}

The above argument also shows that the $2^{-1/8}$ constant in Theorem \ref{u3-system} is sharp.  Indeed, we consider the Heisenberg nilmanifold $G/\Gamma$ in which $G := \R^2 \times (\R/\Z)$ has the group law
$$ (x_1,x_2,z) (x'_1,x'_2,z') := (x_1+x'_1, x_2+x'_2, z+z'+x'_1 x_2)$$
and $\Gamma := \Z^2 \times \{0\}$, and consider a function $f: G/\Gamma \to \C$ of the form
$$ f(x_1,x_2,z) = e(z) \sum_k F_k(x_1) e(k x_2)$$
where the $F_k: \R \to \C$ are functions obeying the periodicity condition
$$ F_k(x_1 + n_1) = F_{k+n_1}(x_1)$$
for all $k,n \in \Z$ and $x_1 \in \R$.  Thus we can in fact write
$$ f(x_1,x_2,z) = e(z) \sum_k F_0(x_1+k) e(k x_2).$$
(The relationship between $f$ and $F_0$ is somewhat similar to that of the \emph{Zak transform} used in signal processing.) The above calculations then show that
$$ \|f\|_{L^2(G/\Gamma)} = \|F_0\|_{L^2(\R)}$$
and
$$ \|f\|_{U^3(G/\Gamma)} = \|F_0\|_{U^3(\R)}.$$
The optimality of the constant $2^{-1/8}$ then follows from the optimality of $2^{-1/8}$ in Theorem \ref{syg}, by taking $F_0$ to be a gaussian (which makes $f$ essentially a theta function).

\begin{remark}\label{total}  The total ergodicity hypothesis was needed in order to reduce to the case when the nilpotent group $G$ was connected.  Without this hypothesis, additional nilsystems can occur which are not contained in the Abramov factor, for which the inequality \eqref{u3-abr} is not reducible to Theorem \ref{syg}, thus requiring a further analysis.  A model example arises by setting $G$ equal to the semidirect product $\Z \ltimes \R^2$, where the generator $e$ of $\Z$ acts by conjugation on $\R^2$ by the formula
$$ e (x,y) e^{-1} := (x,y+x/m)$$
for some integer $m \geq 1$.  If we then set $\Gamma := m\Z \ltimes \Z^2$, and let $\tau \in G$ be the group element $\tau := (\alpha,0) e$ for some irrational $\alpha \in \R$, then $G/\Gamma$ becomes an ergodic (but not totally ergodic) $2$-step nilsystem, with fundamental domain given by $(x,y) e^j$ with $x,y \in [0,1)$ and $j \in \{0,\ldots,m-1\}$, and orbit given by
$$ \tau^n (x,y) e^j \Gamma = ( \{x+n\alpha\}, \{ y + n x / m + \frac{n(n-1)}{2} \alpha / m - \lfloor x+n\alpha \rfloor \frac{n+j}{m} ) e^{(n+j) \mod m} \Gamma.$$
One can verify that the only polynomials of degree at most $2$ on this factor are actually of degree $1$, so the Abramov factor\footnote{On the other hand, if one lifts to the finite extension $G / (m\Z \ltimes m\Z^2)$, then the Abramov factor becomes the entire system.  So one way to extend Theorem \ref{u3-system} to the non-totally-ergodic case is to allow the polynomial $P$ to lie in an extension of $X$, rather than in $X$ itself.  Similar objects have also recently been considered by Szegedy\cite{szeg}.} is equal to the Kronecker factor (i.e. the sets which have trivial behaviour with respect to the $y$ coordinate).  From Theorem \ref{p-system} one must have a bound of the form $\|f\|_{U^3(G/\Gamma)} \leq c \|f\|_{L^2(G/\Gamma)}$ for all $f$ orthogonal to the Abramov factor and some absolute constant $c<1$, but it is not clear to us what this optimal constant is.
\end{remark}

\begin{remark} The above arguments in fact suggest that there should be multiple thresholds; in particular, if $X$ is totally ergodic and $\|f\|_{U^3(X)} > 2^{-d/8} \| f\|_{L^2(X)}$, then $f$ should correlate with a factor generated by the Abramov factor and finitely many Heisenberg nilsystems of dimension less than $2d+1$.  We will not quantify this claim precisely here.
\end{remark}

\section{Threshold for $U^3$ on an interval}\label{u3-n-sec}

We now prove Theorem \ref{u3-n}.  We argue by contradiction.  If the claim failed, then we could find $\eta > 0$, a sequence $N = N_n$ of positive integers, and functions $f = f_n \in L^\infty([N])$ such that $\|f\|_{L^\infty([N]} \leq 1$ and $\|f\|_{U^3([N])} \geq 2^{-1/8}+\eta$, but such that
\begin{equation}\label{fap}
 \langle f, e(P) \rangle_{L^2([N])} = o(1)
\end{equation}
uniformly for all polynomials $P: [N] \to \R/\Z$ of degree at most $2$, where for the purposes of this section, $o(1) = o_{n \to \infty}(1)$ denotes a quantity that goes to zero as $n \to \infty$.  We will show that (possibly after passing to a subsequence) one has
$$ \|f\|_{U^3([N])} \leq 2^{-1/8}+o(1),$$
which will give the required contradiction.

If the $N=N_n$ stay bounded in $n$, then the condition \eqref{fap} and the Plancherel theorem imply that $f$ has an $L^2([N])$ norm of $o(1)$, in which case the claim follows; thus (after passing to a subsequence if necessary) we may assume that $N_n \to \infty$ as $n \to \infty$.

The next step is to apply the arithmetic regularity lemma from \cite{gt-reg}.  We state a form of this lemma suited for our needs:

\begin{lemma}[Arithmetic regularity lemma]\label{arl}  Let ${\mathcal F}: \R^+ \to \R^+$ be a nondecreasing function with ${\mathcal F}(M) \geq M$ for all $M$, let $\eps > 0$, let $N$ be an integer, and let $f \in L^\infty([N])$ with $\|f\|_{L^\infty([N])} \leq 1$.  Then there exists a quantity $M \leq C_{\eps,{\mathcal F}}$ and a decomposition
$$ f = f_{\nil} + f_{\sml} + f_{\unf}$$
into functions $f_{\nil}, f_{\sml}, f_{\unf} \in L^\infty([N])$ obeying the following properties:
\begin{itemize}
\item $f_\nil$ is a $({\mathcal F}(M),N)$-irrational virtual nilsequence of degree $\leq 2$, complexity $\leq M$, and scale $N$.  (We will define this term shortly.)
\item $\|f_\sml\|_{L^2([N])} \leq \eps$.
\item $\|f_\unf\|_{U^3([N])} \leq 1/{\mathcal F}(M)$.
\item $\|f_\nil\|_{L^\infty([N])} \leq 1$.
\end{itemize}
\end{lemma}

\begin{proof}  See \cite[Theorem 1.2]{gt-reg}.  The theorem there is stated for functions taking values in $[0,1]$, but the extension to complex-valued functions bounded in magnitude by $1$ is routine.
\end{proof}

We now pause to recall some definitions from \cite{gt-reg} used in the above lemma; these definitions work in arbitrary degree, but we specialise to the degree $\leq 2$ case for simplicity.  We begin with the concept of a filtered degree $\leq 2$ nilmanifold, which is a slight variant of a $2$-step nilmanifold in which the group is required to be connected (and simply connected), but the commutator group $G_2$ can be replaced with a larger central subgroup of $G$.

\begin{definition}[Filtered nilmanifold]  A \emph{filtered degree $\leq 2$ nilmanifold} consists of the following data:
\begin{itemize}
\item A connected, simply connected $2$-step nilpotent Lie group $G = G_{(0)} = G_{(1)}$, together with a connected, simply connected central subgroup $G_{(2)}$ that contains the commutator subgroup $G_2 = [G,G]$;
\item A discrete cocompact subgroup $\Gamma$ of $G$ such that $\Gamma_{(2)} := \Gamma \cap G_{(2)}$ is cocompact in $G_{(2)}$;
\item A Mal'cev basis ${\mathcal X}$ for $G/\Gamma$ (see \cite{gt-reg} for a definition; we will not need to know the specific properties of such a basis here).
\end{itemize}
We say that a nilmanifold has \emph{complexity at most} $M$ for some $M \geq 2$ if $G$ has dimension at most $M$, and the rationality coefficients of the Mal'cev basis (see \cite[Definition 2.4]{green-tao-nilratner}) is bounded by $M$.  The Mal'cev basis endows $G/\Gamma$ with a metric, the exact definition of which we omit here.

A \emph{polynomial sequence} $g: \Z \to G$ (of degree $\leq 2$) is a sequence of the form $g(n) = g_0 g_1^n g_2^{\binom{n}{2}}$, where $g_0,g_1 \in G$ and $g_2 \in G_{(2)}$.
If $A, N \geq 1$, we say that $g$ is $(A,N)$-irrational if one has
$$ \| \xi_1(g_1) \|_{\R/\Z} \leq A/N$$
and
$$ \| \xi_2(g_2) \|_{\R/\Z} \leq A/N$$
whenever $\xi_1: G \to \R/\Z$ is a nontrivial continuous homomorphism annihilating $\Gamma$ of Lipschitz norm at most $A$ (using the metric on $G/\Gamma$), and $\xi_2: G_{(2)} \to \R/\Z$ is a nontrivial continuous homomorphism annihilating $\Gamma_{(2)}$ of Lipschitz norm at most $A$ (using the induced metric on $G_{(2)}/\Gamma_{(2)}$).  Here $\|x\|_{\R/\Z}$ denotes the distance from $x$ to the nearest integer.
\end{definition}

The precise definition of complexity is not particularly important for our purposes; the only property we need is that for any fixed $M$, there are only finitely many filtered degree $\leq 2$ nilmanifolds of complexity at most $M$, up to isomorphism.

\begin{definition}[Virtual nilsequences]  Let $A, M, N \geq 2$ be integers.  An \emph{$(A,N)$-irrational virtual nilsequence} of degree $\leq 2$, complexity $\leq M$, and scale $N$ is a function $f: \Z \to \C$ of the form
$$ f(n) := F( g(n) \Gamma, n \mod q, n/N )$$
where
\begin{itemize}
\item $G/\Gamma$ is a filtered degree $\leq 2$ nilmanifold of complexity at most $M$;
\item $g: \Z \to G$ is an $(A,N)$-irrational polynomial sequence;
\item $q$ is a positive integer with $q \leq M$; and
\item $F: G/\Gamma \times \Z/q\Z \times \R \to \C$ is a function of Lipschitz norm\footnote{To define this precisely, one needs to specify a metric on $G/\Gamma \times \Z/q\Z \times \R$; the exact choice of metric is not important for our arguments, though, and any reasonable construction will suffice here.} at most $M$.
\end{itemize}
\end{definition}

For further discussion of these concepts we refer the reader to \cite{gt-reg}.

We now apply Lemma \ref{arl} with $\eps = \eps_n := 1/n$ (say) and ${\mathcal F} = {\mathcal F}_n = n {\mathcal F}_0$ for some sufficiently rapid growth function ${\mathcal F}_0$ to be chosen later.  This gives us a quantity $M = M_n$ (which will likely grow quite rapidly in $n$) and a decomposition with the stated properties.  From \eqref{ukp} we have
$$ \|f-f_\nil\|_{U^3([N])} = o(1),$$
so it will suffice to show that
$$ \|f_\nil\|_{U^3([N])} \leq 2^{-1/8}+o(1).$$
From \eqref{pepn} and the Cauchy-Schwarz inequality one has
$$ \langle f_\sml, e(P) \rangle, \langle f_\unf, e(P) \rangle = o(1)$$
uniformly for all quadratic polynomials $P$, so by the triangle inequality and hypothesis we also have
$$ \langle f_\nil, e(P) \rangle = o(1)$$
uniformly for all quadratic polynomials.

By a diagonalisation argument (and choosing ${\mathcal F}_0$ sufficiently rapidly growing depending on $M$), it thus suffices to establish the following:

\begin{proposition}  Let $M \geq 2$ be a fixed integer (independent of $n$).  Suppose that $N = N_n, A = A_n$ are sequences going to infinity, and $f = f_n \in L^\infty([N])$ is a $(A,N)$-irrational virtual nilsequence of degree $\leq 2$, complexity $\leq M$, and scale $N$ such that
$$ \|f\|_{L^\infty([N])} \leq 1$$
and
\begin{equation}\label{fep}
\langle f, e(P) \rangle = o(1)
\end{equation}
uniformly for all quadratic polynomials $P$.  Then
$$ \|f\|_{U^3([N])} \leq 2^{-1/8}+o(1).$$
\end{proposition}

We now prove this proposition.  Fix $M$.  As there are only finitely many isomorphism classes of $G/\Gamma$ with a fixed complexity bound, we may thus (after passing to a subsequence) also fix the nilmanifold $G/\Gamma$.  For similar reasons, we may fix the period $q$.

Write
$$ f(n) := F( g(n) \Gamma, n \mod q, n/N ).$$
As $f$ has $L^\infty([N])$ norm at most $1$, we may clearly (after truncating $F$ if necessary) assume that $F$ is also bounded in magnitude by $1$.

The $U^3([N])$ norm of $f$ can now be computed asymptotically using the \emph{arithmetic counting lemma} from \cite[Theorem 1.11]{gt-reg}, which roughly speaking asserts that the triplet $( g(n) \Gamma, n \mod q, n/N )$ is uniformly distributed in $G/\Gamma \times \Z/q\Z \times [0,1]$ for the purposes of computing arithmetic averages such as the Gowers uniformity norms:

\begin{proposition}\label{propp}  One has
$$ \|f\|_{L^2([N])} = \| F \|_{L^2(G/\Gamma \times \Z/q\Z \times [0,1])} + o(1)$$
and
\begin{align*}
 \|f\|_{U^3([N])}^8 &= \int_{\HK^3([0,1]) \times \HK^3(\Z/q\Z) \times \HK^3(G,G_{(2)})/\HK^3(\Gamma,\Gamma_{(2)})} \\
 &\quad \prod_{\omega \in \{0,1\}^3} {\mathcal C}^{|\omega|} F( x_\omega, y_\omega, z_\omega )\ d\mu_{\HK^3(G,G_{(2)})/\HK^3(\Gamma,\Gamma_{(2)})}(x)\\ &\quad d\mu_{\HK^3(\Z/q\Z)}(y) d\mu_{\HK^3([0,1])}(z) + o(1)
 \end{align*}
where $x = (x_\omega)_{\omega \in \{0,1\}^3}$ and similarly for $y,z$, $\HK^3([0,1])$ is the restriction of $\HK^3(\R)$ to $[0,1]^{\{0,1\}^3}$ with the normalised Lebesgue measure $\mu_{\HK^3([0,1])}$, $\HK^3(G,G_{(2)})$ is defined similarly as to $\HK^3(G)$ but with $G_{(2)}$ taking the place of the commutator group $G_2$, and similarly for $\HK^3(\Gamma,\Gamma_{(2)})$.  ($\Z/q\Z$ will be endowed here with normalised counting measure.)
\end{proposition}

\begin{proof} (Sketch) If $F(x,y,z)$ is independent of the $y,z$ coordinates this follows directly from \cite[Theorem 1.11]{gt-reg} and a routine calculation.  The dependence on $z$ then be inserted by approximating $F$ by a piecewise constant function in $z$, applying \cite[Theorem 1.11]{gt-reg} to each piece, and summing to obtain a Riemann sum that then converges to the required integral.  (Note that $F$ is Lipschitz continuous and thus Riemann integrable.) The dependence on $y$ can be inserted by similarly decomposing the left-hand side into summations over residue classes modulo $q$ and applying the preceding type of computations to each such residue class.
\end{proof}

In view of this proposition, it suffices to establish the estimate
\begin{equation}\label{factor}
\begin{split}
&\int_{\HK^3([0,1]) \times \HK^3(\Z/q\Z) \times \HK^3(G,G_{(2)})/\HK^3(\Gamma,\Gamma_{(2)})} \\
&\quad \prod_{\omega \in \{0,1\}^3} {\mathcal C}^{|\omega|} F( x_\omega, y_\omega, z_\omega )\\
&\quad \ d\mu_{\HK^3(G,G_{(2)})/\HK^3(\Gamma,\Gamma_{(2)})}(x) d\mu_{\HK^3(\Z/q\Z)}(y) d\mu_{\HK^3([0,1])}(z) \\
&\quad \quad \leq 2^{-1} \|F\|_{L^2(G/\Gamma \times \Z/q\Z \times [0,1])} + o(1).
\end{split}
\end{equation}
To show this, we must first convert the information \eqref{fep} into a cancellation property of $F$:

\begin{proposition}[Cancellation property]  Let $G_2 := [G,G]$ and $\Gamma_2 := G_2 \cap \Gamma$, then $G_2/\Gamma_2$ is a torus that acts on $G/\Gamma$, and we have
\begin{equation}\label{gg2}
 \int_{G_2/\Gamma_2} F( g_2 x,y,z )\ d\mu_{G_2/\Gamma_2}(g_2) = o(1)
\end{equation}
uniformly for all $(x,y,z) \in G/\Gamma \times \Z/q\Z \times [0,1]$.
\end{proposition}

\begin{proof}  We will use an argument from \cite[\S 7]{gt-reg}.  Write the left-hand side of \eqref{gg2} as $\tilde F(x,y,z)$, then $\tilde F$ is also Lipschitz continuous (uniformly in $n$) and is also $G_2/\Gamma_2$-invariant.  It will thus suffice to show that the quantity
$$ \| \tilde F \|_{L^2(G/\Gamma \times \Z/q\Z \times [0,1])}^2 = \langle F, \tilde F \rangle_{L^2(G/\Gamma)}$$
is $o(1)$.  Applying (a depolarised variant of) Proposition \ref{propp}, we have
$$ \langle f, \tilde f \rangle_{L^2([N])} = \langle F, \tilde F \rangle_{L^2(G/\Gamma \times \Z/q\Z \times [0,1])} + o(1)$$
and so it suffices to show that
$$ \langle f, \tilde f \rangle_{L^2([N])} = o(1).$$
In view of \eqref{fep} (and the uniform bound on $f$), it suffices to show that for any $\eps > 0$, $\tilde f$ can be approximated uniformly to error $\eps$ by a finite linear combination of quadratic polynomials $e(P)$, where the size and number of coefficients is bounded uniformly in $n$ for fixed $\eps$.

The function $\tilde F$ is $G_2/\Gamma_2$-invariant, and so one can quotient out by this group and reduce to the case when $G_2$ is trivial, i.e. $G$ is abelian.  In this case, $G/\Gamma$ is isomorphic to a torus $(\R/\Z)^m$, and $P: \Z \to \R^m$ is a quadratic polynomial, thus
$$ \tilde f = \tilde F( P(n) \mod \Z^m, n \mod q, n/N ).$$
The claim then follows easily from the Weierstrass approximation theorem.
\end{proof}

By modifying $F$ (and thus $f$) uniformly by $o(1)$, we may now assume that
\begin{equation}\label{cancel}
  \int_{G_2/\Gamma_2} F( g_2 x,y,z )\ d\mu_{G_2/\Gamma_2}(g_2) = 0
\end{equation}
for all $(x,y,z) \in G/\Gamma \times \Z/q\Z \times [0,1]$.  It will then suffice to show the inequality \eqref{factor} (with no $o(1)$ error term) whenever $F$ is a bounded measurable function obeying \eqref{cancel}.

By mimicking the proof of Lemma \ref{Fub}, we see that it then suffices to establish the inequality
$$
\int_{\HK^3(G,G_{(2)})/\HK^3(\Gamma,\Gamma_{(2)})} \prod_{\omega \in \{0,1\}^3} {\mathcal C}^{|\omega|} F( x_\omega )\ d\mu_{\HK^3(G,G_{(2)})/\HK^3(\Gamma,\Gamma_{(2)})}(x) \leq 2^{-1} \|F\|_{L^2(G/\Gamma)}$$
whenever $F$ obeys the cancellation property
\begin{equation}\label{cancel2}
  \int_{G_2/\Gamma_2} F( g_2 x)\ d\mu_{G_2/\Gamma_2}(g_2) = 0
\end{equation}
for all $x \in G/\Gamma$.

Using Fourier decomposition on the torus $G_{(2)}/\Gamma_{(2)}$ (which acts on $G/\Gamma$) as in the previous section, we may assume without loss of generality that there is a character $\xi: G_{(2)}/\Gamma_{(2)} \to \R/\Z$ such that
$$ F(g_2 x) = e(\xi(g_2)) F(x)$$
for all $x \in G/\Gamma$ and $g_2 \in G_{(2)}/\Gamma_{(2)}$.  If $\xi$ annihilates $G_2/\Gamma_2$, then from \eqref{cancel} we see that $F$ is trivial, and the claim follows in this case; so we may assume that $\xi$ is non-trivial on $G_2/\Gamma_2$.

As in the previous section, $F$ is now invariant with respect to the orthogonal complement $\xi^\perp := \{ g_2 \in G_{(2)}/\Gamma_{(2)}: \xi(g_2) = 0 \}$ of $\xi$.  We can then quotient out by that complement and reduce to the case when $G_{(2)}/\Gamma_{(2)} = G_2/\Gamma_2$ is the unit circle $\R/\Z$.  At this point, the nilmanifold $G/\Gamma$ becomes a connected elementary nilmanifold in the language of \cite{host-kra-inf} (and $\HK^3(G,G_{(2)})/\HK^3(\Gamma,\Gamma_{(2)})$ simplifies to $\HK^3(G)/\HK^3(\Gamma)$), and the claim now follows from the results of the previous section.  

\begin{remark} An inspection of the above argument reveals that the condition $\|f\|_{U^3([N])} \geq 2^{-1/8}+\eta$ in Theorem \ref{u3-n} can in fact be relaxed to $\|f\|_{U^3([N])} \geq 2^{-1/8} \| f\|_{L^2([N])}+\eta$.
\end{remark}

\section{Threshold for $U^3$ on cyclic groups}\label{u3-cyclic-sec}

We now use Theorem \ref{u3-n} to prove Theorem \ref{u3-cyclic}.  Fix $f, N, \eta$ as in that theorem.

For each integer $M \geq 1$, let $f^{(M)} = f^{(M)}_n: [MN] \to \C$ be the ``unwrapped'' version of $f$ defined by
$$ f^{(M)}(n) := f(n \mod N)$$
for $n \in [NM]$.  A simple calculation shows that
$$ \|f^{(M)}\|_{U^3([NM])}^8 = \| f \|_{U^3(\Z/N\Z)}^8 + O(1/M).$$
In particular, one has
$$ \|f^{(M)}\|_{U^3([NM])} \geq 2^{-1/8} + \eta/2$$
for all sufficiently large $M$.  Applying Theorem \ref{u3-n}, we conclude that for all sufficiently large $M$, there exists a polynomial $P_M(n) = \alpha_M n^2 + \beta_M n + \gamma_M$ with $\alpha_M,\beta_M,\gamma_M \in \R$ such that
$$ |\langle f^{(M)}, e(P_M) \rangle_{L^2([NM])}| \geq c(\eta)$$
where $c(\eta) > 0$ is independent of $N$ and $M$.  We may normalise so that $\gamma_M = 0$ and $\alpha_M, \beta_M \in [0,1]$, thus
\begin{equation}\label{enm}
 |\E_{n \in [NM]} f(n \mod N) e( - \alpha_M n^2 - \beta_M n )| \geq c(\eta).
\end{equation}
The basic problem here is that $\alpha_M, \beta_M$ are not \emph{a priori} known to be integer multiples of $1/N$, so that $n \mapsto \alpha_M n^2 + \beta_M n$ does not descend to a polynomial in $\Z/N\Z$.  To resolve this, we use the Weyl theory for exponential sums.  Expressing an element $n=[NM]$ as $n=mN+a$ with $m \in [M]-1$ and $a \in [N]$, and using the triangle inequality, one obtains
$$ \E_{a \in [N]} |\E_{m \in [M]-1} e( \alpha_M (mN+a)^2 - \beta_M (mN+a) )| \geq c(\eta)$$
and thus by Cauchy-Schwarz
$$ \E_{a \in [N]} |\E_{m \in [M]-1} e( \alpha_M (mN+a)^2 - \beta_M (mN+a) )|^2 \geq c(\eta)^2.$$
The left-hand side can be expanded as
$$ \E_{m,m' \in [M]-1} \E_{a \in [N]} e( \alpha_M (mN+m'N+2a)(m-m')N - \beta_M (m-m')N ).$$
By the triangle inequality, we thus have
$$ \E_{m,m' \in [M]-1} |\E_{a \in [N]} e( \alpha_M (2a)(m-m')N )| \geq c(\eta)^2.$$
Thus, for at least $c(\eta)^2/2 M^2$ values of $m,m' \in [M]-1$, one has
$$ |\E_{a \in [N]} e( \alpha_M (2a)(m-m')N )| \geq c(\eta)^2/2,$$
which by the geometric series formula implies that
$$ \| \alpha_M N (2a) (m-m') \|_{\R/\Z} \leq C(\eta) / N$$
for such $m,m'$, and some constant $C(\eta) > 0$ depending only on $\eta$.  Applying a lemma of Vinogradov (see \cite[Lemma 3.2]{green-tao-nilratner}), we conclude that for each $M$ and $a \in [N]$, there exists a rational $b_{a,M}/q_{a,M}$ with $|b_{a,M}|,|q_{a,M}| \leq C'(\eta)$ such that
$$ \| 2a \alpha_M N - \frac{b_{a,M}}{q_{a,M}} \|_{\R/\Z} \leq C'(\eta)/NM$$
where $C'(\eta)$ depends only on $\eta$.  By pigeonholing in the $b_{a,M}, q_{a,M}$ and applying the Vinogradov lemma (\cite[Lemma 3.2]{green-tao-nilratner}) again, we can find rationals $b_M/q_M$ with $|b_M|, |q_M| \leq C''(\eta)$ such that
$$ \| \alpha_M N - \frac{b_{M}}{q_{M}} \|_{\R/\Z} \leq C''(\eta)/N^2 M$$
for some $C''(\eta)$ depending only on $\eta$.  By pigeonholing, we may find a rational $b/q$ with $|q| \leq C''(\eta)$ such that
$$ \left|\alpha_M - \frac{b}{cN} \right| \leq \frac{C''(\eta)}{N^2 M}$$
for infinitely many $M$.  We may normalise $q$ to be positive.

By \eqref{enm} pigeonholing, for each such (sufficiently large) $M$ there exists an interval $[m_M N, (m_M+q)N-1]$ such that
$$ \E_{n \in [m_M N, (m_M+q)N - 1]} f(n \mod N) e( - \alpha_M n^2 - \beta_M n )| \geq c(\eta)$$
and thus by translation
$$ |\E_{n \in [qN]} f(n \mod N) e( - \alpha_M n^2 -\beta'_M n)| \geq c(\eta)$$
for some real $\beta'_M$, which we can normalise to lie between $0$ and $1$.  By passing to a subsequence we may assume that $\beta'_M$ converges to a limit $\beta'$, and we conclude that
$$ |\E_{n \in [qN]} f(n \mod N) e( - \frac{b}{qN} n^2 -\beta' n)| \geq c(\eta).$$
We can write $\beta' = \frac{c}{qN} + \frac{\theta}{qN}$ for some integer $c$ and $|\theta| \leq 1$, and so
$$ |\E_{n \in [qN]} f(n \mod N) e( - P(n) ) e( - \theta n / q N )| \geq c(\eta)$$
where $P: \Z \to \R/\Z$ is the polynomial $P(n) := bn^2/qN + cn/qN \mod 1$, which is periodic with period $qN$.

By Urysohn's lemma followed by the Weierstrass approximation theorem (and using the Arzel\'a-Ascoli theorem to get uniform bounds), given any $\eps > 0$ we can approximate the function $x \mapsto e(-\theta x)$ on $[0,1]$ to within $\eps$ in $L^1([0,1])$ norm by a linear combination of exponentials of the form $x \mapsto e(kx)$ for integer $k$, with the size and number of such coefficients bounded uniformly in $\theta$.  Applying this with $\eps = c(\eta)/2$ and using the pigeonhole principle, we see that
$$ |\E_{n \in [qN]} f(n \mod N) e( - P(n) ) e( - k n / q N )| \geq c'(\eta)$$
for some integer $k$ and some $c'(\eta) > 0$ depending only on $\eta$, and the claim follows.


\end{document}